\documentclass[oneside]{amsart}

\newcommand{\RR}{\mathbb{R}}

\usepackage{cite}

\usepackage{amsmath,amssymb,amsthm}
\usepackage{graphicx}

\numberwithin{equation}{section}

\newcommand{\coloneqq}{:=}

\DeclareMathOperator{\supp}{supp}
\DeclareMathOperator{\dist}{dist}
\DeclareMathOperator{\Spec}{spec}
\DeclareMathOperator{\Specess}{spec_\mathrm{ess}}
\DeclareMathOperator{\Specdisc}{spec_\mathrm{disc}}
\DeclareMathOperator{\Ker}{Ker}
\DeclareMathOperator{\Ran}{Ran}

\renewcommand{\tilde}{\widetilde}

\newtheorem{theorem}{Theorem}[section]
\newtheorem{proposition}[theorem]{Proposition}
\newtheorem{lemma}[theorem]{Lemma}
\newtheorem{Corollaires}[theorem]{Corollary}

\theoremstyle{definition}

\begin{document}
%
%

\title[Eigenvalues of Robin Laplacians in infinite sectors]{Eigenvalues of Robin Laplacians in infinite sectors}

\author{Magda Khalile}
\address{Laboratoire de Math\'ematiques d'Orsay, Univ.~Paris-Sud, CNRS, Universit\'e Paris-Saclay, 91405 Orsay, France}
\author{Konstantin Pankrashkin}

\address{Laboratoire de Math\'ematiques d'Orsay, Univ.~Paris-Sud, CNRS, Universit\'e Paris-Saclay, 91405 Orsay, France \&
Laboratoire Poems, INRIA, ENSTA ParisTech, 828, Boulevard des Mar\'echaux, 91762 Palaiseau, France}

\begin{abstract}

For $\alpha\in(0,\pi)$, let $U_\alpha$ denote the infinite planar sector of opening
$2\alpha$,
\[
U_\alpha=\big\{
(x_1,x_2)\in\mathbb R^2:
\big|\arg(x_1+ix_2) \big|<\alpha
\big\},
\]
and $T^\gamma_\alpha$ be the Laplacian in $L^2(U_\alpha)$,
$T^\gamma_\alpha u= -\Delta u$, with the Robin boundary condition
$\partial_\nu u=\gamma u$, where
 $\partial_\nu$ stands for the outer normal derivative and $\gamma>0$.
The essential spectrum of $T^\gamma_\alpha$ does not depend on the angle $\alpha$
and equals $[-\gamma^2,+\infty)$, and the discrete spectrum is non-empty
iff $\alpha<\frac\pi 2$. In this case we show that the discrete spectrum
is always finite and that each individual eigenvalue
is a continous strictly increasing function of the angle $\alpha$.
In particular, there is just one discrete eigenvalue
for $\alpha \ge \frac{\pi}{6}$. 
As $\alpha$ approaches $0$, the number of discrete eigenvalues
becomes arbitrary large and is minorated by $\kappa/\alpha$ with a suitable $\kappa>0$,
and the $n$th eigenvalue $E_n(T^\gamma_\alpha)$ of $T^\gamma_\alpha$ behaves as
\[
E_n(T^\gamma_\alpha)=-\dfrac{\gamma^2}{(2n-1)^2 \alpha^2}+O(1)
\]
and admits a full asymptotic expansion in powers of $\alpha^2$.
The eigenfunctions are exponentially localized near the origin.
The results are also applied to $\delta$-interactions on star graphs.
\end{abstract}

\maketitle

\section{Introduction}

For $\alpha\in(0,\pi)$, let $U_\alpha$ denote the infinite sector of opening
$2\alpha$,
\[
U_\alpha=\big\{
(x_1,x_2)\in\mathbb R^2:
\big|\arg(x_1+ix_2) \big|<\alpha
\big\}.
\]
In the present paper, we are interested in the spectral properties
of the associated Robin Laplacian, to be denoted $T^\gamma_\alpha$, which is defined
as follows: for $\gamma>0$, the operator
$T^\gamma_\alpha$ acts in $L^2(U_\alpha)$ as
\[
u\mapsto -\Delta u:=-\Big(
\dfrac{\partial^2}{\partial x_1^2}
+
\dfrac{\partial^2}{\partial x_2^2}
\Big)u,
\]
on the functions $u$ satisfying the boundary condition
\[
\dfrac{\partial u}{\partial \nu}=\gamma u \text{ at } \partial U_\alpha,
\]
where $\nu$ stands for the outer unit normal. More precisely, $T^\gamma_\alpha$
is defined as the unique self-adjoint operator in $L^2(U_\alpha)$
corresponding to the sesquilinear form
\begin{equation}
      \label{talph}
t^\gamma_\alpha(u,u)=\int_{U_\alpha} |\nabla u|^2 dx_1 dx_2 -\gamma\int_{\partial U_\alpha} |u|^2 ds, \quad u\in H^1(U_\alpha),
\end{equation}
where $ds$ is the one-dimensional Hausdorff measure.

During the last years, the spectral analysis of Robin Laplacians attracted
a considerable attention. As shown in \cite{lacey, ref13}, the spectral properties
of such operators are sensible to the regularity of the boundary,
and the case of smooth domains has been studied in detail~\cite{hk,hkr,pp15,2d,pp16}.
On the other hand, only partial results are available
for domains with a non-smooth boundary, cf. \cite{bp16,ref2,conv}.

The study of the above operator $T^\gamma_\alpha$ has several interesting aspects
from the point of view of the existing results.
First, it can be viewed as the simplest non-smooth domain in two dimensions
and depending in an explicit way on the single geometric parameter $\alpha$.
Second, its spectral properties play an important role in the study
of more general non-smooth domains in the strong coupling limit, see \cite{bp16,clr,ref13},
as the sectors $U_\alpha$ exhaust the whole family of possible tangent cones to the boundary
in two dimensions. Third, the domain $U_\alpha$ and its boundary are non-compact,
which may potentially lead to quite unusual spectral properties,
such as the presence of an infinite discrete spectrum, and the respective
studies in higher dimensions \cite{ref14} do not extend directly to the planar
case: we will see below that the results are actually quite different.
In fact, the only spectral result on $T^\gamma_\alpha$ available
in the existing literature is as follows, see \cite{ref13}:
\begin{equation}
     \label{eq-tga}
\inf \Spec T^\gamma_\alpha=\begin{cases}
-\gamma^2, & \alpha \ge \dfrac{\pi}{2},\\
-\dfrac{\gamma^2}{\sin^2\alpha}, & \alpha < \dfrac{\pi}{2},
\end{cases}
\end{equation}
and for $\alpha<\frac{\pi}{2}$ the value indicated is an eigenvalue
with an explicitly known eigenfunction $\exp(-\gamma x_1/\sin \alpha)$.
The aim of the present work is to provide a more detailed
spectral analysis.

Our main results are as follows. First, as the essential spectrum of
$T^\gamma_\alpha$ does not depend on the angle $\alpha$
and equals $[-\gamma^2,+\infty)$, see Theorem~\ref{SPEC},
it follows  from \eqref{eq-tga} that the discrete spectrum is non-empty
if and only if $\alpha<\frac{\pi}{2}$, i.e. if and only
if the sector is strictly
smaller than the half-plane.
(It is worth noting that a similar geometric effect
appears for other classes of differential operators, see e.g. \cite{babich,kamot}.)
In Theorem~\ref{thm31} we show that the discrete spectrum
is always finite, which is a non-trivial result
due to the non-compactness of the boundary. In subsection~\ref{monotone}
we obtain more detailed results: in Theorem~\ref{increase}
we show that each individual eigenvalue
is a strictly increasing continuous function of the angle $\alpha$
and, moreover, that there is just one discrete eigenvalue for $\alpha \ge \frac{\pi}{6}$,
see Theorem~\ref{thmpi6}. In section~\ref{small1}
we discuss the behavior of the discrete eigenvalues for small $\alpha$.
We show that the $n$th eigenvalue $E_n(T^\gamma_\alpha)$ behaves as
\[
E_n(T^\gamma_\alpha)=-\dfrac{\gamma^2}{(2n-1)^2 \alpha^2}+O(1),
\quad \alpha\to 0,
\]
see Corollary~\ref{corol43}, and, moreover, admits
an asymptotic expansion up to any order with respect to
the powers of $\alpha^2$, see Theorem~\ref{full}. The
number of discrete eigenvalues becomes arbitrary large
and is minorated by $\kappa/\alpha$, $\kappa>0$,
as $\alpha$ approaches $0$, see Corollary~\ref{cor42}.
In Theorem~\ref{agmon} we show that the associated
eigenfunctions are localized, in a suitable sense, near the vertex,.
Appendices \ref{appa}--\ref{appd} contain some technical details of  the proof.
In Appendix~\ref{sec6} we apply the results obtained to the study
of Schr\"odinger operators with $\delta$-interactions
supported by star graphs, and show that such operators always
have a finite discrete spectrum, which was missing in the existing literature.

Our proofs are mostly variational and based on the min-max characterization
of the essential spectrum and the eigenvalues. The proof of the finiteness of the discrete spectrum
uses an idea proposed in \cite{ref12} for a different operator
involving a similar geometry, while the continuity and the monotonicity
of the eigenvalues are established using a suitable change of variables.
The asymptotics for small $\alpha$ is based on the well-known Born-Oppenheimer
strategy \cite{bo,raybook} and, similar to various problems involving
small parameters \cite{vbn,HJ,r1,r2,naz1,raybook}, on a reduction to a one-dimensional
effective operator, which in our case acts in $L^2(\mathbb R_+)$ as
\[
f \mapsto \Big( -\dfrac{d^2}{dr^2}-\dfrac{1}{4r^2}-\dfrac{1}{\alpha r}\Big)f
\]
with a suitable boundary condition at the origin, while an additional
work is required due to the singularity of the potential.
We remark that the use of an improved Hardy
inequality \cite{ref9} allows one to perform the reduction
in a rather direct way, without any preliminary
localization argument for the associated eigenfunctions. 
The properties of the eigenfunctions are studied with the help of the standard
Agmon-type approach using a suitable decomposition of the domain~\cite{agmon}.

We remark that in the present paper we are not discussing
the (non-)existence of eigenvalues embedded into the continuous spectrum.
Some partial information can be easily obtained, for example,
the recent Rellich-type result \cite{bonnet} implies the absence
of positive embedded eigenvalues  for $\alpha\ge \frac{\pi}{2}$,
but a specific separate study is needed in order to cover all possible
cases. This will be discussed elsewhere.

\section{Preliminaries}

\subsection{Min-max principle}\label{sminmax}

Recall first the min-max principle giving a variational characterization of eigenvalues.
Let $A$ be a self-adjoint operator acting in a Hilbert space $\mathcal H$ of infinite dimension.
We assume that $A$ is semibounded from below, $A\ge -c$, $c\in \mathbb R$, and denote
\[
\Sigma:=\begin{cases}
\inf \Specess A, & \text{ if }\Specess A\ne \emptyset,\\
+\infty, & \text{ if }\Specess A=\emptyset.
\end{cases}
\]
By $E_j(A)$ we denote its $j$th eigenvalue when ordered in the non-decreasing
order and counted with multiplicities.
The domain and the form domain of $A$ will be denoted by $D(A)$ and $Q(A)$, respectively, and by
$a:\,Q(A)\times Q(A)\to \mathbb C$
we denote the associated sesquilinear form. Recall that
$Q(A)$ is a Hilbert space when considered with the scalar product
$\langle u,v\rangle_a:=a(u,v)+(c+1)\langle u,v\rangle$.
The following result is a standard tool of the spectral theory, see e.g. \cite[Section XIII.1]{rs4}.
\begin{theorem}[Min-max principle]
Let $n \in \mathbb N$ and $Q$ be a dense subset of the Hilbert space $Q(A)$. Let $\Lambda_n(A)$
be the $n$th \emph{Rayleigh quotient} of $A$, which is defined by
\[
\Lambda_n(A):=\sup_{\psi_1,...,\psi_{n-1} \in \mathcal H}\inf_{\substack{\varphi \in Q,\varphi \neq 0\\ \varphi \perp \psi_j, j=1,...,n-1}}\frac{a(\varphi,\varphi)}{\langle \varphi,\varphi \rangle}\equiv \inf_{\substack{G\subset Q \\ \dim G=n}} \sup_{\substack{\varphi \in G\\\varphi \neq 0}} \frac{a(\varphi,\varphi)}{\langle \varphi,\varphi \rangle},
\]
then one and only one of the following assertions is true:
\begin{enumerate}
\item $\Lambda_n(A)<\Sigma$ and $E_n(A)=\Lambda_n(A)$.
\item $\Lambda_n(A)=\Sigma$ and $\Lambda_m(A)=\Lambda_n(A)$ for all $m\ge n$.
\end{enumerate}
\end{theorem}

We remark that most of the subsequent constructions are heavily based
on estimates for the Rayleigh quotients
of various operators.

\subsection{Robin Laplacian on the half-line}\label{halfline}

For $\gamma>0$ denote by $B_\gamma$ the self-adjoint operator acting in $L^2(\mathbb R_+)$
by
\[
B_\gamma u=-u'', \quad  D(B_\gamma)=\big\{u\in H^2(\mathbb R_+):  -u'(0)=\gamma u(0)\big\}.
\]
One easily checks that $\Specess B_\gamma=[0,+\infty)$ and that the unique eigenvalue is
$E_1(B_\gamma)=-\gamma^2$ with $u(t)=e^{-\gamma t}$ the associated eigenfunction.
Remark that the sesquilinear form for $B_\gamma$ is
\[
b_\gamma(u,u)=\int_0^{+\infty} \lvert u'(t) \rvert^2 dt -\gamma \lvert u(0) \vert ^2,
\quad
D(b_\gamma)=H^1(\mathbb R_+),
\]
hence
\begin{eqnarray}
\int_0^{+\infty} \lvert u'(t) \rvert^2 dt -\gamma \lvert u(0) \vert ^2 \geq - \gamma^2 \int_{0}^{+\infty} \lvert u(t) \rvert ^2 dt,
\quad
u\in H^1(\mathbb R_+).
\label{2.1.1}
\end{eqnarray}

\subsection{Robin Laplacian on an interval}\label{interval}

For $L>0$ and $\gamma\in \mathbb R$, consider the operator $B_{L,\gamma}$ in $L^2(-L,L)$ acting by
\[
B_{L,\gamma}u=-u'',\quad
D(B_{L,\gamma})=\big\{u \in H^2(-L,L): \quad -u'(-L)=\gamma u(-L), \quad u'(L)=\gamma u(L)\big\}.
\]
Remark that the associated sesquilinear form $b_{L,\gamma}$ is
\[
b_{L,\gamma}(L) (u,u)=\int_{-L}^L \lvert u'(t) \rvert ^2 dt -\gamma \lvert u(-L)\rvert^2 -\gamma  \lvert u(L)\rvert^2,
\quad u \in H^1(-L,L).
\]
The operator $B_{L,\gamma}$ has a compact resolvent and its spectrum is purely discrete
and consists of the simple eigenvalues $E_j(L,\gamma)$, $j\in \mathbb N$,
numbered in the increasing order. Remark that due to the scaling we have
\begin{eqnarray}
E_j(L,\gamma)=\frac{1}{L^2} E_j(1,\gamma L), \quad j\in \mathbb N.
\label{2.1.2}
\end{eqnarray}
An easy application of the min-max principle shows that
the maps $\mathbb R\ni \gamma \mapsto E_j(1,\gamma)$ are continuous,
in particular $E_j(1,0)$ coincides with the $j$th eigenvalue of the Neumann Laplacian on $(-1,1)$,
hence
\begin{align}
\label{2.1.3}
E_1(1,0)&=0, \\
\label{2.1.4}
E_2(1,0)&=\frac{\pi^2}{4}.
\end{align}
Moreover, as follows from the computations of Appendix in \cite{ref2}, one has
$E_1(1,\gamma)=-k^2$ with $k>0$ being the solution to $k \tanh k= \gamma$, hence,
the function $\mathbb R\ni \gamma \mapsto E_1(1,\gamma)$ is real-analytic.

By \cite[Proposition A.3]{ref2} the following assertions hold true:
\begin{itemize}
\item for any $\gamma>0$ there holds $E_1(L,\gamma)<0$, and the associated eigenfunction
is
\[
\Phi_{L}(\gamma,t)=\cosh\big(\sqrt{-E_1(L,\gamma)}t\big),
\]
\item the inequality $E_2(L,\gamma)<0$ holds if and only if $\gamma L>1$.
\item For large $L$ there holds
\begin{align}
\label{2.1.5}
E_1(L,\gamma) &=-\gamma^2-4 \gamma^2 e^{-2\gamma L}+O(Le^{-4L}), \\
\label{2.1.6}
E_2(L,\gamma)&=-\gamma^2+4 \gamma^2e^{-2 \gamma L}+O(Le^{-4L}).
\end{align}
\end{itemize}
Let us introduce some quantites to be used later in the text.  For $\gamma \in\mathbb R_+$, denote
\[
m(\gamma):= \sqrt{-E_1(1,\gamma)},
\]
then we have
\[
 E_1(L,\gamma)=-\dfrac{m(\gamma L)^2}{L^2},
\quad
\Phi_{L}(\gamma,t)=\cosh\Big(\dfrac{m(\gamma L)t}{L}\Big),
\]
and a simple direct computation gives
\[
\int_{-L}^L \big\lvert \Phi_{L}(\gamma,t)\big\rvert^2 dt=L\left ( \frac{\sinh\big(2m(\gamma L)\big)}{2m(\gamma L)}+1 \right ) .
\]
The first eigenfunction $\widetilde \Phi_{L}(\gamma,\cdot)$ of $B_{\gamma,L}$,
chosen positive and normalized, is then given by
\[
\widetilde \Phi_{L}(\gamma,t)=C_L(\gamma) \Phi_{L}(\gamma,t) \quad\text{with}\quad C_L(\gamma)\coloneqq \frac{1}{\sqrt{L}}\left ( \frac{\sinh\big(2m(\gamma L)\big)}{2m(\gamma L)}+1 \right ) ^{-\frac{1}{2}}.
\]
For any fixed $L$ and $t$ the functions $\gamma \mapsto C_L(\gamma)$ and $\gamma \mapsto \Phi_{L}(\gamma,t)$ are smooth, and, by direct computation,
\begin{gather*}
\dfrac{\partial}{\partial\gamma} C_L(\gamma)=-\sqrt{L} m'(\gamma L) \left (\frac{\sinh\big(2m(\gamma L)\big)}{2m(\gamma L)}+1 \right )^{-\frac{3}{2}} \left (\frac{\cosh(2m(\gamma L))}{2m(\gamma L)}-\frac{\sinh\big(2m(\gamma L)\big)}{(2m(\gamma L))^2} \right ), \\
\dfrac{\partial}{\partial\gamma} \Phi_{L}(\gamma,t)=m'(\gamma L) t \sinh\Big(m(\gamma L) \frac{t}{L}\Big).
\end{gather*}
Moreover, as $\gamma\mapsto E_1(L,\gamma)$ is analytic and simple, we can compute its first derivative
in the standard way, namely, consider the implicit equation satisfied by $\Phi_{L}(\gamma,\cdot)$ and $E_1(L,\gamma)$:
\begin{align}
\label{eqimp1}
-\frac{\partial^2}{\partial t^2}\Phi_{L}(\gamma,t)&=E_1(L,\gamma)\Phi_{L}(\gamma,t),\quad t\in(-L,L)\\
\label{eqimp2}
\pm \frac{\partial}{\partial t}\Phi_{L}(\gamma,\pm L)& = \gamma \Phi_{L}(\gamma,\pm L).
\end{align}
We take the derivative of \eqref{eqimp1} with respect to $\gamma$ and, after multiplying by $\Phi_{L}(\gamma,\cdot)$, we integrate it over $(-L,L)$. After two integrations by part we obtain the following equality
\begin{multline}
\label{eqimp2bis}
\int_{-L}^L \partial_{\gamma}\Phi_{L}(\gamma,t) \left(-\frac{\partial^2}{\partial t^2}\Phi_{L}(\gamma,t) \right)dt\\
- \Big[\Big(\partial_t\partial_{\gamma}\Phi_{L}(\gamma,t)\Big)\Big(\Phi_{L}(\gamma,t) \Big)\Big]_{-L}^L + \Big[\Big(\partial_{\gamma}\Phi_{L}(\gamma,t)\Big)\Big(\partial_t\Phi_{L}(\gamma,t)\Big)\Big]_{-L}^L \\
= \partial_\gamma E_1(L,\gamma) \int_{-L}^L\lvert \Phi_{L}(\gamma,t)\rvert^2 dt +E_1(L,\gamma) \int_{-L}^L \partial_\gamma\Phi_{L}(\gamma,t) \Phi_{L}(\gamma,t) dt.
\end{multline}
We now take the derivative of \eqref{eqimp2} with respect to $\gamma$ and we get
\begin{align*}
\partial_\gamma\partial_t \Phi_{L}(\gamma,L)&=\Phi_{L}(\gamma,L)+\gamma\partial_\gamma \Phi_{L}(\gamma,L),\\
-\partial_\gamma\partial_t \Phi_{L}(\gamma,-L)&=\Phi_{L}(\gamma,-L)+\gamma\partial_\gamma \Phi_{L}(\gamma,-L).
\end{align*}
After replacing these expressions in \eqref{eqimp2bis} we finally have
\begin{align}
\label{2.1.12}
\frac{\partial}{\partial \gamma} E_1(L,\gamma)=-\dfrac{\lvert \Phi_{L}(\gamma,-L)\rvert^2+\lvert \Phi_{L}(\gamma,L)\rvert^2}{\lVert \Phi_{L}(\gamma,\cdot)\rVert^2_{L^2(-L,L)}}=-\frac{2\cosh^2\left(m(\gamma L)\right)}{L\left(\dfrac{\sinh\big(2m(\gamma L)\big)}{2m(\gamma L)}+1\right)}.
\end{align}
In particular,
\begin{equation}
  \label{tmp5}
\frac{\partial}{\partial \gamma} E_1(1,\gamma)\Big\lvert_{\gamma=0}=-1.
\end{equation}
Furthermore, due to the preceding consideration, the following representation is valid:
\begin{proposition}{\label{Prop2.5}}
There exists  $\phi \in C^\infty(\mathbb R_+)\mathop{\cap}L^\infty(\mathbb R_+)$ such that
$E_1(1,\gamma)=-\gamma+\gamma^2\phi(\gamma)$ for all $\gamma \in \mathbb R_+$.
\end{proposition}
\begin{proof}
Define $\phi:\mathbb R_+\to \mathbb R$ by
$\phi(\gamma)=\gamma^{-2} \big(E_1(1,\gamma)+\gamma\big)$.
Due to \eqref{2.1.5} and to the equality $L^2 E_1(L,\gamma)= E_1(1,\gamma L)$,
the function $\phi$ is bounded at infinity.
For $\gamma$ near $0$ one has, due to \eqref{2.1.3} and \eqref{tmp5}
and due to the analyticity, $E_1(1,\gamma)=-\gamma+O(\gamma^2)$,
which shows that $\phi$ is bounded near $0$. As $\phi$ is
continuous, the result follows.
\end{proof}

\subsection{Robin Laplacians in sectors: first properties}

The following theorem is a starting point for our considerations.
The results are essentially known for the specialists,
in particular, the points (b) and (c) were discussed in \cite{ref13},
but, to our knowledge, it was never stated explicitly so far.
We prefer to give a complete proof in Appendix~\ref{appa2}
in order to keep the presentation self-contained.

\begin{theorem}\label{SPEC}
For any $\alpha \in (0,\pi)$ and  any $\gamma>0$ the sesquilinear form $t^\gamma_\alpha$ given by \eqref{talph}
is closed and semibounded from below, hence, the associated operator $T^\gamma_\alpha$ is self-adjoint
in $L^2(U_\alpha)$.
Furthermore,
\begin{itemize}
\item[(a)] $\Specess T^\gamma_\alpha=[-\gamma^2,+\infty)$ for any $\alpha \in (0,\pi)$.
\item[(b)] If $\alpha\in\big(0,\frac{\pi}{2}\big)$, then
\[
E_1(T^\gamma_\alpha)=-\dfrac{\gamma^2}{\sin^2\alpha},
\]
and $u(x_1,x_2)=\exp\big(-\gamma x_1/\sin \alpha\big)$
is an associated eigenfunction.
\item[(c)] for $\alpha \in \big[\frac{\pi}{2},\pi\big)$, the discrete spectrum of $T^\gamma_\alpha$
is empty.
\end{itemize}
\end{theorem}

Therefore, the discrete spectrum of $T^\gamma_\alpha$
is non-empty if and only if $\alpha<\frac{\pi}{2}$, \emph{which will be assumed
in the rest of the paper}. Our principal aim is to obtain a more detailed information
on the number of discrete eigenvalues
and on their behavior with respect to the angle $\alpha$.

Remark that the domain $U_\alpha$ is invariant by dilations,
hence, the operator $T^\gamma_\alpha$
is unitarily equivalent to $\gamma^2 T^1_\alpha$.
Therefore, in what follows
we restrict our attention to the operator
\[
T_\alpha\coloneqq T_\alpha^1.
\]

\section{Qualitative spectral properties}

Our first objective is to show that the discrete spectrum of $T_\alpha$
is finite. One should remark that the result is dimension-dependent
in the sense that Robin Laplacians on cones may have an infinite discrete spectrum
in higher dimensions, as shown in \cite{bpp,ref14}.
Later, in section \ref{monotone}, we prove that the eigenvalues
of $T_\alpha$ are monotone continuous functions of $\alpha$.

\subsection{Reduction by parity}\label{parity}

We start with a decomposition of $T_\alpha$
due to the symmetry of the sector.
Consider the upper half $U^+_\alpha$ of $U_\alpha$,
$U_\alpha^+=U_\alpha \mathbin{\cap} (\mathbb R\times \mathbb R_+)$,
and the unitary map
\begin{gather*}
\mathcal U:
L^2(U_\alpha)\ni u \mapsto (g,h) \in L^2(U_\alpha^+)\oplus L^2(U_\alpha^+),\\
g(x_1,x_2):=\frac{u(x_1,x_2)+u(x_1,-x_2)}{\sqrt{2}},\quad
h(x_1,x_2):=\frac{u(x_1,x_2)-u(x_1,-x_2)}{\sqrt{2}}.
\end{gather*}
By direct computation, for $u \in D(t_\alpha)$ one has
$t_\alpha(u,u)=t_\alpha^N(g,g)+t_\alpha^D(h,h)$ with
\[
t_\alpha^N(g,g)=\int_{U_\alpha^+} \Big| \nabla g(x_1,x_2) \Big|^2 dx-\int_{\mathbb R^+} \Big|\,
g\left(\frac{x_2}{\tan \alpha},x_2\right)\Big|^2 \frac{dx_2}{\sin \alpha},
\quad
g\in H^1(U_\alpha^+),
\]
and $t_\alpha^D$ is given by the same expression but acts on the smaller domain
\[
D(t_\alpha^D)=\{h \in H^1(U_\alpha^+):  h(\cdot,0)=0\},
\]
and $D(t^N_\alpha)=P_1\mathcal U D(t_\alpha)$
and $D(t^D_\alpha)=P_2\mathcal U D(t_\alpha)$,
where $P_j:L^2(U_\alpha^+)\oplus L^2(U_\alpha^+)\to L^2(U_\alpha^+)$
is the projection onto the $j$the component.
Hence, if $T_\alpha^N$ and $T_\alpha^D$ are the self-adjoint operators acting in $L^2(U^+_\alpha)$ and associated with
$t_\alpha^N$ and $t_\alpha^D$, respectively, then, by construction,
$T_\alpha= \mathcal U^* (T_\alpha^N \oplus T_\alpha^D) \mathcal U$,
and it is sufficient to study separately the spectra of $T_\alpha^D$ and $T_\alpha^N$.

Let us show first that
\begin{equation}
  \label{tda}
\inf \Spec T^D_\alpha\ge -1.
\end{equation}
To see this, consider the half-plane
\begin{equation}
    \label{eq-paa}
P_\alpha=\big\{(x_1,x_2)\in \mathbb R^2,\quad x_1\geq \frac{x_2}{\tan\alpha}\big\}
\end{equation}
and remark that if one takes $u\in D(t^D_\alpha)$ and denotes by $\widetilde u$
its extension by zero to $P_\alpha$, then
\[
t^D_\alpha(u,u)=q_{P_\alpha}(\widetilde u,\widetilde u),
\]
where
\[
q_{P_\alpha}(u,u)=\int_{\mathbb R}\int_{\frac{x_2}{\tan\alpha}}^{+\infty} \big| \nabla u(x_1,x_2)\big|^2 dx-
\int_{\mathbb R} \Big| u\left(\frac{x_2}{\tan\alpha},x_2\right)\Big|^2\frac{dx_2}{\sin\alpha},
\quad
D(q_{P_\alpha})=H^1(P_\alpha).
\]

If $Q_{P_\alpha}$ is the self-adjoint operator associated with $q_{P_\alpha}$ and acting in $L^2(P_\alpha)$, then
$\inf\Spec T^D_\alpha\geq \inf\Spec Q_{P_\alpha}$. On the other side, by applying a rotation
one sees that $Q_{P_\alpha}$ is unitarily equivalent to
$Q=B_1\otimes 1 + 1\otimes L$ acting in $L^2(\mathbb R\times\mathbb R_+)\simeq L^2(\mathbb R)\otimes L^2(\mathbb R_+)$,
where $B_1$ is defined in subsection~\ref{halfline} and $L$ is the free Laplacian in $L^2(\mathbb R)$.
In particular, $\Spec Q=\overline{\Spec B_1 + \Spec L}=[-1,+\infty)$, which proves \eqref{tda}.
Therefore, we have
\[
\Spec T_\alpha\mathop{\cap}(-\infty,-1)=\Spec T^N_\alpha\mathop{\cap}(-\infty,-1).
\]
Furthermore, in view of Theorem \ref{SPEC} we have
\[
\Specess T^N_\alpha \subset \Specess T_\alpha=[-1,+\infty),
\]
in particular,
\begin{equation}
     \label{specdisc}
\Specdisc T_\alpha=\Specdisc T^N_\alpha\mathop{\cap}(-\infty,-1),
\end{equation}
and the eigenvalue multiplicities are preserved. It also follows that
all eigenfunctions of $T_\alpha$ associated with the discrete eigenvalues
are even with respects to $x_2$.

\subsection{Finiteness of the discrete spectrum}\label{fini}

\begin{theorem}\label{thm31}
The discrete spectrum of $T_\alpha$ is finite for any $\alpha\in(0,\frac{\pi}{2})$.
\end{theorem}

\begin{proof}
In view of \eqref{specdisc}, it is sufficient to show that the operator $T^N_\alpha$
has only a finite number of eigenvalues in $(-\infty,-1)$.
During the proof, if $A$ is a self-adjoint operator associated to a semi-bounded from below
sesquilinear form $a$ and $\lambda\in \mathbb R$, we denote
by $N(A,\lambda)$ or $N(a,\lambda)$ the number of the eigenvalues (counting the multiplicities) of $A$ in $(-\infty,\lambda)$ for $\Specess A \cap (-\infty,\lambda)=\emptyset$,
and set $N(A,\lambda)=+\infty$ otherwise.
The proof scheme is inspired by \cite[Theorem 2.1]{ref12}.
The idea is to perform a dimensional reduction in order to compare the operator with a one-dimensional one and to conclude using a Bargmann-type estimate.

We first introduce a decomposition of $U_\alpha^+$.
Let $\chi_0$ and $\chi_1$ be smooth real-valued
functions defined on $\mathbb R_+$ such that
\[
\chi_0(t)=1 \text{ for } \quad 0< t  <1, \quad
\chi_0(t)=0 \text{ for } \quad  t >2, \quad
\chi_0^2+\chi_1^2=1.
\]
For $R>1$, to be determined later, consider the functions $\chi_{0,R}$ and $\chi_{1,R}$ defined on $U^+_\alpha$ by 
\[
\chi_{j,R}(x_1,x_2)\coloneqq\chi_j\Big(\frac{x_2}{R}\Big),\quad j=0,1,
\]
and the following subdomains of $U_\alpha^+$,
\[
A_R=\{ (x_1,x_2) \in U_\alpha^+ :  0<x_2<2R\}, \quad
B_R =\{(x_1,x_2) \in U_\alpha^+ : x_2> R \},
\]
see Fig~\ref{fig:Covering}.
We get easily 
\[t_\alpha^N(u,u)=t_\alpha^N(\chi_{0,R}u,\chi_{0,R}u) + t_\alpha^N(\chi_{1,R}u, \chi_{1,R}u) - \lVert u \nabla\chi_{0,R} \rVert ^2_{L^2(U_{\alpha})} - \lVert u \nabla \chi_{1,R} \rVert ^2_{L^2(U_{\alpha})},\]
for all $u \in D(t_\alpha^N)$, so that
\begin{multline}
t_\alpha^N(u,u) \geq  \int_{A_R} \Big (\lvert \nabla(u \chi_{0,R}) \rvert ^2 - V_R\lvert u \chi_{0,R}\rvert ^2  \Big )dx - \int_{\partial U_\alpha \cap \partial A_R} \lvert u \chi_{0,R} \rvert ^2 ds  \\
            + \int_{B_R} \Big (\lvert \nabla(u \chi_{1,R}) \rvert ^2 - V_R\lvert u \chi_{1,R} \rvert ^2  \Big ) dx - \int_{\partial U_\alpha^+ \cap \partial B_R} \lvert u \chi_{1,R} \rvert ^2 ds,
\label{3.1.3}
\end{multline}
where
\[
V_R(x_1,x_2)= \displaystyle \sum_{j=0,1} \big\lvert \nabla\chi_{i,R}(x_1,x_2)\big\rvert ^2
\equiv
\dfrac{1}{R^2}\displaystyle \sum_{j=0,1} \Big\lvert \chi'_i\Big(\dfrac{x_2}{R}\Big)\Big\rvert^2,
\quad
\|V_R\|_\infty \le \dfrac{C}{R^2}, \quad
C:=\|\chi'_0\|^2_\infty+\|\chi'_1\|^2_\infty.
\]
\begin{figure}
  \centering
		\includegraphics[width=0.60\textwidth]{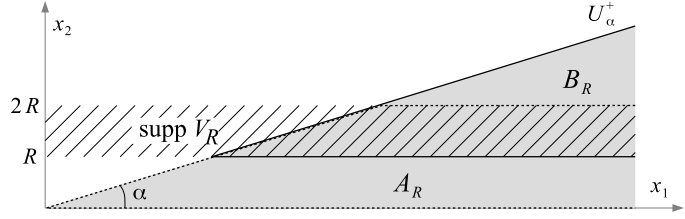}
	\caption{\label{fig:Covering} Covering of $U_\alpha^+$ by the domains $A_R$ (surrounded by the dash line) and $B_R$
	(surrounded by the solid line). The support of $V_R$ is hatched.}
\end{figure}

We define the following sesquilinear forms:
\begin{align*}
q_{A_R}(u,u)&= \int_{A_R} \Big( \lvert \nabla u \rvert ^2 - V _R \lvert u \rvert ^2 \Big ) dx - \int_{\partial U_\alpha \cap\partial A_R} \lvert u \rvert ^2 ds,\\
D(q_{A_R})&=\big \{ u \in H^1(A_R): u(\cdot,2R)=0  \big \},\\
q_{B_R}(u,u)&= \int_{B_R} \Big( \lvert \nabla u \rvert ^2 - V _R \lvert u \rvert ^2 \Big ) dx - \int_{\partial U_\alpha^+ \cap \partial B_R} \lvert u \rvert ^2 ds,\\ 
D(q_{B_R})&=\big \{ u \in H^1(B_R): u(\cdot,R)=0 \big \}.
\end{align*}
Due to
\[
u\chi_0 \in D(q_{A_R}), \quad
u\chi_1 \in D(q_{B_R}), \quad
\lVert u\chi_0\rVert^2_{L^2(A_R)}+\lVert u\chi_1\rVert^2_{L^2(B_R)}=\lVert u\rVert ^2_{L^2(U^+_\alpha)},
\]
the min-max principle and \eqref{3.1.3} give for for any $n\in\mathbb N$:
\begin{align*}
\Lambda_n(T^N_\alpha)&\geq \min_{\substack{G\subset D(t^N_\alpha)\\\dim(G)=n}}\max_{\substack{u\in G\\ u\neq 0}} \frac{q_{A_R}(u\chi_0,u\chi_0)+q_{B_R}(u\chi_1,u\chi_1)}{\lVert u\chi_0\rVert ^2_{L^2(A_R)}+\lVert u\chi_1\rVert^2_{L^2(B_r)}} \\
&\geq \min_{\substack{G\subset D(q_{A_r}\oplus q_{B_R})\\\dim(G)=n}}\max_{\substack{(u_0,u_1)\in G\\ (u_0,u_1)\neq (0,0)}} \frac{q_{A_R}(u_0,u_0)+q_{B_R}(u_1,u_1)}{\lVert u_0\rVert ^2_{L^2(A_R)}+\lVert u_1\rVert^2_{L^2(B_r)}}\\
&=\Lambda_n(Q_{A_R}\oplus Q_{B_R}),
\end{align*}
where $Q_{A_R}$ and $Q_{B_R}$ are the self-adjoint operators acting respectively
in $L^2(A_R)$ and $L^2(B_R)$ and produced by the forms $q_{A_R}$
and $q_{B_R}$. Then,
\begin{align}
\label{3.1.4}
N(T_\alpha^N, - 1) \leq N(q_{A_R},- 1) + N(q_{B_R}, - 1 ).
\end{align}

Let us first estimate $N(q_{A_R},- 1)$. We consider the following two domains:
\[
A^0_R=\Big \{(x_1,x_2) \in A_R: 0<x_1 < \frac{2R}{\tan\alpha} \Big \}, \quad
A^1_R=\Big \{(x_1,x_2) \in A_R: x_1>\frac{2R}{\tan\alpha} \Big \},
\]
and the sesquilinear forms
\begin{align*}
q_{A^0_R}(u,u)&=\int_{A^0_R} \Big ( \lvert \nabla u \rvert ^2 - \lvert u \rvert ^2 V_R \Big) dx - \int_{\partial A^0_R \cap \partial U_{\alpha}} \lvert u \rvert ^2 ds, \quad  D(q_{A^0_R})= H^1(A^0_R), \\
q_{A_R^1} (u) &= \int_{A^1_R} \Big (\lvert \nabla u \rvert ^2 - V_R \lvert u \rvert ^2 \Big ) dx, \quad  D(q_{A^1_R}) = \Big \{ u \in H^1(A^1_R): u(\cdot,2R)=0 \Big \}.
\end{align*}
By  the min-max principle we have
\begin{align}
\label{3.1.5}
N(q_{A_R},- 1) \leq N(q_{A^0_R}, -1 ) + N(q_{A_R^1},- 1).
\end{align}
On one hand, the operator $Q_{A^0_R}$ associated to $q_{A_R^0}$ has a compact resolvent, which implies
\begin{align}
\label{3.1.6}
N(Q_{A_R^0},- 1) < + \infty.
\end{align}
On the other hand, the operators $Q_{A^1_R}$ associated to $q_{A^1_R}$
can be represented as
\[
Q_{A^1_R} = T_N\otimes 1 + 1 \otimes T_{DN}(2R),
\]
where $T_N$ is the Neumann Laplacian in $L^2(2R/ \tan\alpha, + \infty)$ and $T_{DN}(2R)$ is the self-adjoint operator
in $L^2(0,2R)$ associated with the sesquilinear form
\[
t_{DN}(u,u)=\int_{0}^{2R} \Big (\lvert u' (t) \rvert ^2 - V_R(t) \rvert u (t)\rvert ^2 \Big)dt,\quad u\in D(t_{DN})=\Big\{u \in H^1(0,2R), \text{ }u(2R)=0\Big\}.
\]
One has $\inf \Spec T_N = 0$ and $\inf \Spec T_{DN}(2R) \geq - \lVert V_R \rVert_{\infty}$.
Setting $R_1 =\sqrt{C}$ and taking $R>R_1$ implies $\inf\Spec T_{DN}(2R) \geq -1$ and then
\begin{align}
\label{3.1.7}
N(Q_{A^1_{R}},- 1) = 0 \text{ for } R> R_1,
\end{align}
and we conclude by \eqref{3.1.5}, \eqref{3.1.6} and \eqref{3.1.7} that $N(q_{A_R}, - 1) < + \infty$, for all $R> R_1$.

Now let us estimate $N(q_{B_R}, - 1 )$ for $R>R_1$. Let us introduce the sesquilinear form
\[
q^R_{P_\alpha}(u,u)=\int_{P_\alpha}\Big ( \lvert \nabla u \rvert^2 -V_R\lvert u\vert^2\Big) dx - \int_{\mathbb R} \Big\lvert u \Big(\frac{x_2}{\tan \alpha},x_2\Big)\Big \rvert^2 \frac{dx_2}{\sin \alpha}, \quad D(q^R_{P_\alpha})=H^1(P_\alpha),
\]
where $P_\alpha$ is the half-plane given by \eqref{eq-paa}, then $N(q_{B_R},-1)\leq N(q^R_{P_\alpha},-1)$
by the min-max principle.
If we make an anti-clockwise rotation of angle $\frac{\pi}{2}-\alpha$ of $P_\alpha$, then we obtain the half-plane $\mathbb R_+\times\mathbb R$, and the operator $Q^R_{P_\alpha}$ associated with $q^R_{P_\alpha}$ is then unitarily equivalent to
the operators $Q^R$ associated with the sesquilinear form
\[
q^R(u,u)=\int_{\mathbb R_+ \times \mathbb R}\left(\lvert\nabla u\rvert^2-\tilde V_R\lvert u\rvert^2\right)dx-\int_{\mathbb R}\big\lvert u(0,x_2)\big\rvert^2dx_2,\quad u\in D(q^R)=H^1(\mathbb R_+\times \mathbb R),
\]
where
\begin{align*}
\tilde V_R(x_1,x_2)&=V_R(x_1\sin\alpha+x_2\cos\alpha,x_2\sin\alpha-x_1\cos\alpha)
\equiv v_R(x_2\sin \alpha-x_1\cos \alpha ),\\
v_R(t)&=\dfrac{1}{R^2}\displaystyle \sum_{j=1,2} \Big\lvert \chi'_i\Big(\dfrac{t}{R}\Big)\Big\rvert^2
\end{align*}
One has, for $u\in D(q^R)$,
\[
q^R(u,u)=q(u,u)-\int_{\mathbb R_+ \times \mathbb R} \tilde V_R \lvert u \rvert ^2 dx, 
\]
where the operator associated to $q$ is $Q=B_1 \otimes 1 + 1 \otimes L$ with $B_1$ defined in subsection~\ref{halfline} and $L$
the free Laplacian in $L^2(\mathbb R)$. Let us consider the orthogonal projections $\Pi$ and $P$ in $L^2(\mathbb R_+ \times \mathbb R)$,
\begin{align*}
\Pi u(x_1,x_2)&=\sqrt{2} e^{-x_1} \psi(x_2),\quad \psi(x_2)=\sqrt{2}\int_{\mathbb R_+} u(t,x_2) e^{-t} dt, \\
P u &= u - \Pi u.
\end{align*}
Remark that $\Pi=\pi\otimes 1$, where $\pi$ is the spectral projector of $B_1$ on  $\{-1\}$.
For $u \in D(q^R)$ there holds $\Pi u,Pu\in D(q^R)$, and
\[
q^R(u,u)=q^R(\Pi u,\Pi u)+ q^R(P u,Pu ) -2\Re \int_{\mathbb R_+\times\mathbb R} \tilde V_R \overline{\Pi u} Pu  dx,
\]
as $q(\Pi u,P u)=0$ by the spectral theorem. Writing 
\[
W_R(x_2)=2 \int_{\mathbb R_+} e^{-2x_1 } \tilde V_R(x_1,x_2) dx_1
\]
we have
\[
 q^R(\Pi u,\Pi u)= \int_\mathbb R \Big( \lvert \psi'(x_2) \rvert ^2 - W_R(x_2) \lvert \psi (x_2) \rvert ^2 \Big ) dx_2 -
\lVert \psi \rVert ^2_{L^2(\mathbb R)}.
\]
By the spectral theorem applied to $B_1$, for a.e. $x_2\in \mathbb R$ one has,
\[
\displaystyle \int_ {\mathbb R_+}\Big\lvert\frac{\partial P u}{\partial x_1}(x_1,x_2)\Big\rvert ^2 dx_1-
\big\lvert u(0,x_2)\big\rvert ^2 \geq 0,
\]
and, finally,
\[
 q^R(P u,Pu) \geq \int_{\mathbb R} \int_{\mathbb R_+} \Big ( \Big\lvert \frac{\partial P u}{\partial x_2} \Big\rvert ^2 - \tilde V_R \lvert P u \rvert ^2 \Big ) dx.
\]
For any $\epsilon \in\mathbb R_+$ one can estimate
\[
2 \left \lvert\int_{\mathbb R_+ \times \mathbb R } \overline{\Pi u} P u \tilde V_R dx \right \rvert \leq \epsilon \lVert Pu \rVert^2_{L^2(\mathbb R_+\times \mathbb R)} +\frac{1}{\epsilon} \lVert \Pi u \tilde V_R\rVert^2_{L^2(\mathbb R_+\times\mathbb R)}.
\]
Then, using the equality $\lVert \psi \rVert_{L^2(\mathbb R)}= \lVert \Pi u \rVert_{L^2(\mathbb R_+ \times \mathbb R) }$, we get 
\[
 q^R(u,u) \geq \int_{\mathbb R} \Big (\lvert \psi '(x_2) \rvert ^2 - Z_R(x_2) \lvert \psi(x_2) \rvert ^2 \Big ) dx_2 -\lVert \Pi u \rVert ^2_{L^2(\mathbb R_+ \times \mathbb R)} - \Big(\epsilon+ \frac{C}{R^2}\Big) \lVert P u \rVert ^2_{L^2(\mathbb R_+ \times \mathbb R)},
\]
where 
\[
Z_R(x_2)= W_R(x_2)+\frac{1}{\epsilon} \int_{\mathbb R_+} 2 e^{-2 x_1} \tilde V_R ^2 (x_1,x_2) dx_1.
\]
We can choose $R_2>R_1$ and $\epsilon >0$ such that $\epsilon + C/R_2 ^2 \leq 1$,
then for $R>R_2$ one arrives at
\begin{align}
\label{3.1.9}
\ q^{R}(u,u) \geq \int_{\mathbb R} \Big (\lvert \psi ' \rvert ^2 - Z_{R} \lvert \psi \rvert ^2\Big ) dx_2  -\lVert u \rVert ^2 _{L^2(\mathbb R_+ \times \mathbb R)}.
\end{align}
We introduce the sesquilinear form 
\[
a^R(\psi,\psi)= \int_{\mathbb R} \Big (\big\lvert \psi'(x_2) \big\rvert ^2  - Z_{R}(x_2)
\big\lvert \psi (x_2) \big \rvert ^2 \Big ) dx_2, \quad  u\in H^1(\mathbb R),
\]
then, by \eqref{3.1.9} and the min-max principle we have
\begin{align}
\label{3.1.10}
N( q^{R},- 1 ) \leq N(a^R, 0), \quad R>R_2.
\end{align}
In order to show that the number of negative eigenvalues of $a^R$ is finite, we want to use a Bargmann-type estimate, see e.g. \cite[Eq.(8)]{ref4}:
\[
N(a^R,0) \le 2+\int_{\mathbb R} \lvert x_2 \rvert Z_{R}(x_2) dx_2.
\]
We can write, using the fact that $\supp \chi_0'\cup \supp \chi_1' \subset [1,2]$, 
\[
\int_{\mathbb R} \lvert x_2 \rvert Z_{R}(x_2) dx_2 = \int_{\mathbb R_+ } 2 e^{-2 x_1} \left( \int_{\frac{R+ x_1 \cos \alpha}{\sin \alpha}} ^{\frac{2R + x_1 \cos \alpha}{\sin\alpha}} \lvert x_2 \rvert \Big( \tilde V_{R}(x_1,x_2) + \frac{1}{\epsilon}\tilde V_{R} ^2(x_1,x_2) \Big)dx_2 \right)dx_1.
\]
Using the boundedness of $\tilde V_R$ we finally get the following upper bound:
\[
\int_{\mathbb R} \lvert x_2 \rvert Z_{R}(x_2) dx_2 \leq \frac{C}{R \sin^2 \alpha}
\Big(1 + \frac{C}{R^2 \epsilon }\Big)\int_{\mathbb R_+} e^{-2 x_1 } (R+2x_1\cos\alpha )dx_1 < + \infty.
\]
Hence, $N(a^R,0)<+\infty$ and \eqref{3.1.10} implies that $N( q^{R},-1)<+\infty$ for $R>R_2$. By \eqref{3.1.4} we conclude that $N(T^N_\alpha,-1)<+\infty$.
\end{proof}

\subsection{Continuity and monotonicity with respect to the angle}\label{monotone}

Let us discuss first the monotonicity of the Rayleigh quotients of $T_\alpha$
with respect to $\alpha$.
\begin{proposition}\label{prop-mon1}
For any $n\in\mathbb N$ the function
$(0,\frac{\pi}{2})\ni\alpha \mapsto \Lambda_n(T_\alpha)$
is non-decreasing and continuous.
\end{proposition}

\begin{proof}
In view of the constructions of subsection~\ref{parity}
it is sufficient to show the result for the Rayleigh quotients $\Lambda_n(T^N_\alpha)$
of $T^N_\alpha$.
We denote by $\tilde t^N_\alpha$ the sesquilinear form obtained after the anti-clockwise rotation of angle $\frac{\pi}{2}-\alpha$ of $U_\alpha^+$. Then $\tilde t^N_\alpha$ is unitarily equivalent to $t^N_\alpha$ and we have
\[
\tilde t_\alpha^N(g,g)=\int_{\tilde U_\alpha^+} \lvert \nabla g \rvert^2 dx - \int_{\mathbb R_+}\lvert g\rvert^2(0,x_2)dx_2,
\]
for all $g\in H^1(\tilde U_\alpha^+)$, where $\tilde U_\alpha^+=\Big\{(x_1,x_2)\in(\mathbb R_+)^2, x_1\leq x_2\tan\alpha\Big\}$. After the scaling $t=x_2\tan\alpha$ and writing  $\tilde g(x_1,t)=g(x_1,\frac{t}{\tan\alpha})$ we have
\[
\tilde t_\alpha^N(g,g)=\int_{\tilde U_{\frac{\pi}{4}}^+}\left ( \lvert \partial_{x_1} \tilde g (x_1,t)\rvert^2+\tan^2\alpha \lvert \partial_t \tilde g(x_1,t)\rvert^2 \right)\frac{dx_1dt}{\tan\alpha}-\int_{\mathbb R^+} \lvert \tilde g(0,t)\rvert ^2\frac{dt}{\tan\alpha},
\]
and $\lVert g \rVert ^2_{L^2(\tilde U_\alpha^+)}=\displaystyle\frac{1}{\tan\alpha} \lVert \tilde g\rVert^2_{L^2(\tilde U^+_{\frac{\pi}{4}})}$. 
Then, we can define 
\[
q_\alpha(v,v)=\int_{\tilde U_{\frac{\pi}{4}}^+}\left ( \lvert \partial_{x_1} v (x_1,t)\rvert^2+\tan^2\alpha \lvert \partial_t  v(x_1,t)\rvert^2 \right)dx_1dt-\int_{\mathbb R^+} \lvert v(0,t)\rvert ^2dt,
\]
for $v\in D(q_\alpha)=H^1(\tilde U_{\frac{\pi}{4}}^+)$ and $Q_\alpha$ the associated operator
in $L^2(\tilde U_{\frac{\pi}{4}}^+)$.
By construction, we have $\Lambda_n(T_\alpha)=\Lambda_n(Q_\alpha)$ .
The dependence  of $Q_\alpha$ on $\alpha$ only appears through the coefficient
$(\tan \alpha)^2$, which gives the result due to the min-max principle.
\end{proof}

Let us now obtain a stronger result for the eigenvalues.

\begin{theorem}\label{increase}
Assume that for some $n\ni\mathbb N$ and $\alpha_n\in (0,\frac\pi 2)$ the operator $T_{\alpha_n}$
has at least $n$ discrete eigenvalues, then $T_\alpha$ has at least $n$ discrete eigenvalues
for all $\alpha<\alpha_n$, and the function $(0,\alpha_n)\ni \alpha\mapsto E_n(T_\alpha)$
is strictly increasing.
\end{theorem}

\begin{proof}
If $T_{\alpha_n}$ has at least $n$ discrete eigenvalues, then by the min-max principle one
has $\Lambda_n(T_{\alpha_n})<-1$, which by Proposition~\ref{prop-mon1} implies
$\Lambda_n(T_\alpha)<-1$ for all $\alpha<\alpha_n$, hence $\Lambda_n(T_\alpha)$
is the $n$th discrete eigenvalue of $T_\alpha$ by the min-max principle.
The weak monotonicity and continuity also follow from Proposition~\ref{prop-mon1}.

Let us show the strict monotonicity of the eigenvalues.
Let $ \alpha_1$, $\alpha_2$ such that $\alpha_1<\alpha_2$
and $E_n(T_{\alpha_2})<-1$. We continue using the notation
of the proof of Proposition~\ref{prop-mon1}, then $E_n(Q_{\alpha_2})=E_n(T_{\alpha_2})<-1$,
and we need to show the strict inequality
$E_n(Q_{\alpha_1})<E_n\big(Q_{\alpha_2})$.
For all $\varphi \in H^1(\tilde U^+_{\frac{\pi}{4}})$ one has, with
$\kappa=\tan^2 \alpha_2-\tan^2 \alpha_1>0$,
\begin{align}
\label{sc1}
\frac{q_{\alpha_1}(\varphi,\varphi)}{\lVert \varphi \rVert ^2_{L^2(\tilde U_{\frac{\pi}{4}}^+)}}= \frac{q_{\alpha_2}(\varphi,\varphi)}{\lVert \varphi \rVert ^2_{L^2(\tilde U^+_{\frac{\pi}{4}})}} - \kappa \frac{\displaystyle \int_{\tilde U_{\frac{\pi}{4}}^+}\lvert \partial_t \varphi \rvert^2 dx_1dt}{\lVert \varphi\rVert^2_{L^2(\tilde U_{\frac{\pi}{4}}^+)}}.
\end{align}
Let $\varphi_1,...,\varphi_n$ be an orthonormal basis
in
$\mathcal K_n:=\sum_{k=1}^n
\Ker\big(Q_{\alpha_2}-E_k(Q_{\alpha_2})\big)$.
On one hand, for all $\varphi\in \mathcal K_n\setminus\{0\}$ we have
\[
\frac{q_{\alpha_2}(\varphi,\varphi)}{\lVert \varphi \rVert ^2_{L^2(\tilde U^+_{\frac{\pi}{4}})}}\le E_n(Q_{\alpha_2}).
\]
On the other hand, using the min-max principle and \eqref{sc1} we have
\[
E_n(Q_{\alpha_1})\leq \sup_{\substack{\varphi \in \mathcal K_n \\ \varphi \neq 0}}\frac{q_{\alpha_1}(\varphi,\varphi)}{\lVert \varphi \rVert ^2_{L^2(\tilde U_{\frac{\pi}{4}}^+)}}
\le
E_n(Q_{\alpha_2})-\kappa \inf_{\substack{\varphi \in \mathcal K_n\setminus\{0\} \\ \varphi \neq 0}} \frac{\displaystyle \int_{\tilde U_{\frac{\pi}{4}}^+}\lvert \partial_t \varphi \rvert^2 dx_1dt}{\lVert \varphi\rVert^2_{L^2(\tilde U_{\frac{\pi}{4}}^+)}}.
\]
Assume that $E_n(Q_{\alpha_1})=E_n(Q_{\alpha_2})$, then the second term on the right-hand side is zero.
As the unit ball of $\mathcal K_n$ is compact, there must exist
$\varphi\in\mathcal K_n$ with $\|\varphi\|=1$
such that
\[
\int_{\tilde U_{\frac{\pi}{4}}^+}\lvert \partial_t \varphi \rvert^2 dx_1dt=0,
\]
i.e. $\partial_t \varphi(x_1,t)=0$. Then $\varphi$ depends on the $x_1$ variable only, but as
$\varphi \in L^2(\tilde U^+_{\frac{\pi}{4}})$ we necessarily have 
$\varphi=0$, which contradicts the normalization $\|\varphi\|=1$.
\end{proof}

Another important corollary is as follows: 

\begin{Corollaires}\label{corol123}
Assume that for some $\alpha_1$ the operator $T_{\alpha_1}$
has a unique discrete eigenvalue, then $T_\alpha$
has a unique discrete eigenvalue for all $\alpha\in \big(\alpha_1, \frac{\pi}{2}\big)$.
\end{Corollaires}

A natural candidate for $\alpha_1$ is $\frac{\pi}{4}$. In fact,
the respective operator $T_{\frac\pi 4}$ admits a separation of variables
and is unitarily equivalent to $B_1\otimes 1+1 \otimes B_1$
with $B_1$ given in subsection~\ref{halfline}.
Hence, $T_{\frac\pi 4}$ has a unique discrete eigenvalue $(-2)$,
which shows

\begin{Corollaires}\label{oneeigenV}
For  $\alpha \in [\frac{\pi}{4}, \frac{\pi}{2})$,
the operator $T_\alpha$ admits a unique discrete eigenvalue.
\end{Corollaires}

In fact, we can obtain a better estimate:

\begin{theorem}\label{thmpi6}
The operator $T_\alpha$ has a unique discrete eigenvalue
for  $\alpha \in \big[\frac\pi 6,\frac \pi 2\big)$.
\end{theorem}

\begin{proof}
In view of Corollary~\ref{corol123} it is sufficient to consider $\alpha=\frac\pi6$.
We continue using the domain $\Tilde U^+_\alpha$
from the proof of Proposition \ref{prop-mon1}.
Let $Q^L_\alpha$ be the Laplacian in $L^2(\Tilde U^+_{\frac{\pi}{6}})$
with the Robin boundary condition $\partial u/\partial \nu=u$
at $x_1=0$, the Neumann boundary condition
at the line $x_1=x_2\tan \alpha$ and with the Neumann boundary condition
at the both sides of the lines $x_2=L$ and $x_1=L/\sqrt 3$,
then by the min-max principle for any $L>0$ and $k\in \mathbb N$ one has the inequality
\begin{equation}
     \label{eq-lll1}
\Lambda_k(T_{\frac \pi 6})=\Lambda_k(Q_{\frac \pi 6})\ge \Lambda_k(Q_{\frac \pi 6}^L).
\end{equation}
Let us argue by contradiction. Assume that $\Lambda_2(T_{\frac \pi 6})<-1$, then it follows
from \eqref{eq-lll1} that
\begin{equation}
       \label{eq-lll2}
\limsup_{L\to+\infty}\Lambda_2(Q_{\frac \pi 6}^L)< -1.
\end{equation}
Remark that $Q_{\frac \pi 6}^L=A_{1,L}\oplus A_{2,L}\oplus A_{3,L}$, where
$A_{1,L}$ is the Laplacian in the triangle
\[
\Omega_L=\Big\{
(x_1,x_2): 0<x_1< \frac{x_2}{\sqrt{3}},\, 0<x_2<L
\Big\}
\]
with the Robin boundary condition at $x_1=0$ and with the Neumann boundary condition
on the other two sides,
the operator $A_{2,L}$ in the Laplacian in the half-strip $\Pi_L=(0,L/\sqrt 3)\times(L,\infty)$
with the Robin boundary condition at $x_1=0$ and with the Neumann boundary condition
at the remaining part of the boundary, and $A_{3,L}$ is the Neumann Lalplacian
in $\Tilde U^+_{\frac\pi 6}\setminus \overline{\Omega_L\cap\Pi_L}$, see Fig.\ref{prop36}(a).
\begin{figure}
	\centering
\begin{tabular}{cc}
\begin{minipage}[c]{0.30\textwidth}
\begin{center}
\includegraphics[width=0.6\textwidth]{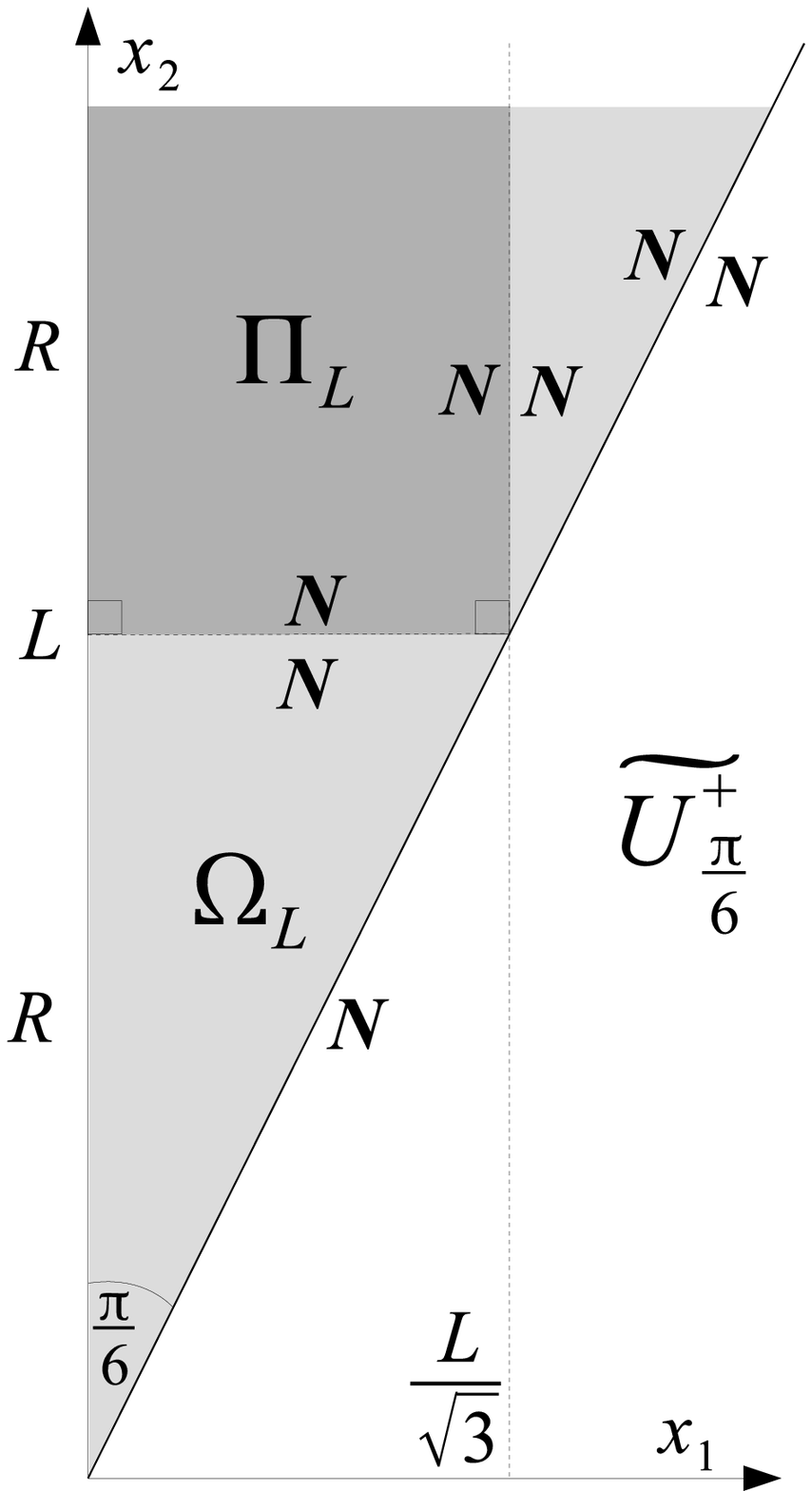}
\end{center}
\end{minipage}
&
\begin{minipage}[c]{0.30\textwidth}
\begin{center}
\includegraphics[width=0.8\textwidth]{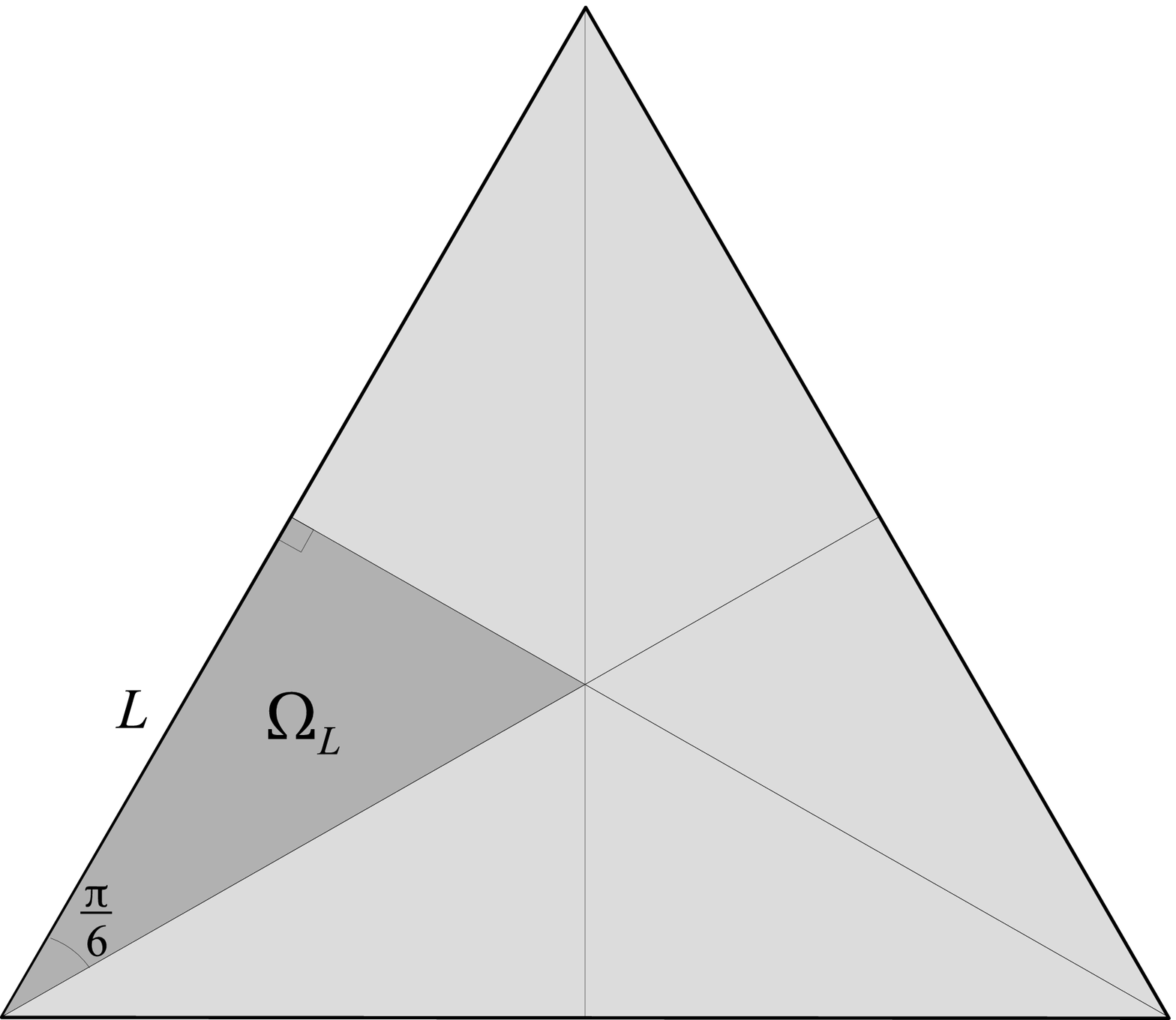}
\end{center}
\end{minipage}\\
(a) & (b)
\end{tabular}
\caption{\label{prop36} (a) Decomposition of $\tilde U^+_{\frac{\pi}{6}}$.
The symbols $R$ and $N$ indicate
respectively the Robin and Neumann boundary conditions.
(b) The triangle $2L \Theta$.}
\end{figure}

The spectrum of $A_{2,L}$ can be easily estimated using the separation of variables, and then
$A_{2,L}\ge -1+o(1)$ for $L\to+\infty$. Furthermore, $A_{3,L}\ge 0$, and Eq.~\eqref{eq-lll2}
implies
\begin{equation}
       \label{eq-lll3}
\limsup_{L\to+\infty}\Lambda_2(A_{1,L})< -1.
\end{equation}
Remark that each eigenfunction of $A_{1,L}$ can be extended, using the reflections
with respect to the Neumann sides, to an eigenfunction of the Laplacian $K_{2L}$
with the Robin boundary condition $\partial u/\partial \nu=u$ on the equilateral
triangle $2L \Theta$ composed from six copies on $\Omega_L$, see Figure~\ref{prop36}(b),
where $\Theta$ is an equilateral triangle of side length $1$.

Therefore, $\Lambda_2(A_{1,L})=\Lambda_2(K_{2L,s})$, where
$K_{2L,s}$ is the restriction of $K_{2L}$ to the functions which are invariant
under the reflections with respect to the medians
and with respect to the rotations by $\frac{2\pi}{3}$.
The eigenvalues of $K_{2L}$ in the limit $L\to+\infty$
were analyzed in \cite[Section 7]{mcc}. In particular, the first
eigenfunction
of $K_{2L}$ has the above-mentioned symmetries, hence,
$\Lambda_1(K_{2L})=\Lambda_1(K_{2L,s})$.
On the other hand, the second eigenvalue of $K_{2L}$
is double-degenerate, and no associated eigenfunction
has the required symmetries: there is just one eigenfunction, noted
$T^{0,1}_s$ in \cite[Section 7]{mcc}, which is even with respect to one of the medians,
but it does not possess the other symmetries.
Therefore, $\Lambda_2(K_{2L,s})\ge \Lambda_4(K_{2L})$ for large $L$.
On the other hand, it is shown in \cite[Subsection 7.4]{mcc}
that $\liminf_{L\to+\infty}\Lambda_j(K_{2L})\ge -1$
for $j\ge 4$, which contradicts \eqref{eq-lll3}.
This contradiction shows that the inequality
$\Lambda_2(T_{\frac \pi 6})\equiv\Lambda_2(Q_{\frac \pi 6})<-1$
is not possible, and $T_{\frac \pi 6}$ has a unique discrete eigenvalue.
\end{proof}

\section{Asymptotics of eigenvalues for small angle}\label{small1}

The present section is devoted to the study of the discrete spectrum
of $T_\alpha$ as $\alpha$ tends to $0$.

\subsection{First order asymptotics}
We are going to show first the following result giving
an estimate for the Rayleigh quotients of $T_\alpha$:
\begin{theorem}\label{Thprinc}
There exists $\alpha_0\in(0,\frac{\pi}{2})$ and  $\mathcal C>0$ such that for all $\alpha\in(0,\alpha_0)$ and for all $n\in\mathbb N$ there holds
\[
\Big\lvert\, \Lambda_n(T_\alpha)+\frac{1}{(2n-1)^2\alpha^2}\Big\rvert \leq \mathcal C.
\]
\end{theorem}

Before passing to the proof let us discuss the most important consequences.
Recall that $N(T_\alpha,-1)$ is the number of discrete eigenvalues of $T_\alpha$,
and it is finite in virtue of Theorem~\ref{thm31}.

\begin{Corollaires}\label{cor42}
There exists $\kappa>0$ such that $N(T_\alpha,-1)\geq \kappa/\alpha$
as $\alpha$ is small. In particular, $N(T_\alpha,-1)$ tends to $+\infty$ as $\alpha$ tends to $0$.
\end{Corollaires}
\begin{proof}
For all $\alpha\in(0,\alpha_0)$ and any $n\in\mathbb N$ one has 
\[
\Lambda_n(T_\alpha)\leq \mathcal C - \frac{1}{(2n-1)^2\alpha^2}.
\]
By the min-max principle, the number $\Lambda_n(T_\alpha)$ is the $n$th eigenvalue
iff it is strictly less than $(-1)$.
Notice that the right-hand side is smaller than $(-1)$ for all $n<n_\alpha$,
\[
n_\alpha:=\displaystyle\frac{1}{2} \left(\displaystyle\frac{1}{\alpha\sqrt{1+\mathcal C}}+1\right),
\]
 and then $N(T_\alpha,-1)$ is not smaller than the integer part
of $n_\alpha-1$.
\end{proof}

Another obvious corollary is
\begin{Corollaires}\label{corol43}
For any $n\in\mathbb N$ there holds
\[
E_n(T_\alpha)=-\frac{1}{(2n-1)^2 \alpha^2} +O(1) \text{ as $\alpha$ tends to $0$.}
\]
\end{Corollaires}
In Theorem~\ref{full} below we show a stronger result
that $E_n(T_\alpha)$ admits a full asymptotic expansions
in powers of $\alpha^2$.

\subsubsection{Model one-dimensional operator}\label{op1D}

The main idea of the proof is to compare the operator
$T_\alpha$ with some one-dimensional operator.
Namely, for $a>0$ consider the following operator acting in $L^2(\mathbb R_+)$:
\[
(H_a v)(r)=\Big(-\frac{d^2}{dr^2} -\frac{1}{4r^2}-\frac{1}{a r} \Big) v(r),
\quad v\in C^\infty_c(\mathbb R_+)
\]
and $h_a$ be the associated sesquilinear form,
\[
h_a(v,v)=\langle v, H_a v\rangle_{L^2(\mathbb R_+)}
=\int_0^\infty \Big(|v'|^2-\dfrac{|v|^2}{4r^2}
-\dfrac{|v|^2}{ar} \Big)dr,
\quad v\in C^\infty_c(\mathbb R_+).
\]
Denote by $H^\infty_a$ the Friedrichs extension of $H_a$.
It is known (see Appendix~\ref{appb}) that its essential spectrum
equals $[0,+\infty)$
and that the discrete spectrum consists of the simple negative eigenvalues
\[
\mathcal E_n(a):=E_n(H^\infty_a)=-\dfrac{1}{(2n-1)^2 a^2}, \quad n\in \mathbb N,
\]
while the respective eigenfunctions $\psi_n$ are
\[
\psi_n(r)=\sqrt{r} e^{-\frac{r}{(2n-1)a}} L_{n-1}\Big(\dfrac{2r}{(2n-1)a}\Big)
\]
with $L_m$ being the Laguerre polynomials.

\subsubsection{Polar coordinates}\label{polarcoord}

Denote
$V_\alpha=(0,+\infty)\times(-\alpha,\alpha)$
and
define a unitary operator
$\mathcal U \colon L^2(U_\alpha,dx_1dx_2)\to L^2(V_\alpha,drd\theta)$,
\[
{\mathcal U }u(r,\theta)= r^{\frac{1}{2}}u(r\cos\theta,r\sin\theta).
\]

In order to use the min-max principle for the Robin Laplacians
we need the following density result, which is quite
standard (see  Appendix~\ref{appa} for the proof):
\begin{lemma}\label{density}
The set
\[
\mathcal F =\Big\{u\in C^\infty(\overline{U_\alpha}): \text{ there exist } R_1,R_2 >0 \text{ such that } u=0 \text{ for } \lvert x\rvert < R_1\text{ and for } \lvert x\rvert >R_2\Big\}
\]
is dense in $H^1(U_\alpha)$.
\end{lemma}
We have then
\[
\mathcal U(\mathcal F)=\mathcal G:=\big\{v\in C^\infty (\overline {V_\alpha}): \text{ } \exists R_1,R_2>0
\text{ such that } u(r,\theta)=0 \text{ for } r < R_1 \text{ and for } r >R_2\big\}.
\]
We define the new sesquilinear form $q_\alpha(v,v)\coloneqq t_\alpha(\mathcal U^*v,\mathcal U^*v)$,
\[
q_\alpha(v,v)= \int_{V_\alpha}\left ( \lvert v_r\rvert^2 -\frac{1}{4} \frac{\lvert v \rvert ^2}{r^2} +\frac{\lvert v_\theta\rvert^2}{r^2}\right) drd\theta-\int_{\mathbb R_+} \left (\frac{\lvert v(r,\alpha)\rvert^2}{r}+\frac{\lvert v(r,-\alpha)\rvert^2}{r} \right) dr, \text{ } v\in\mathcal G,
\]
defined initially on $\mathcal G$. As the set $\mathcal F$
is dense in the form domain of $T_\alpha$,
the operator $Q_\alpha$ corresponding to the closure of $q_\alpha$
writes as $Q_\alpha=\mathcal U T_\alpha \mathcal U^*$, and
$\Lambda_n(T_\alpha)=\Lambda_n(Q_\alpha)$ for all $n\in\mathbb N$.

Recall that the numbers $E_j(a,b)$ are defined in subsection~\ref{interval}.
The next proposition is a direct application of the min-max principle.
\begin{proposition}
For all $v \in \mathcal G$ we have
\[q_\alpha(v,v) \geq \int_{V_\alpha}
\Big(
 \lvert v_r\rvert^2 -\frac{1}{4}\frac{\lvert v \rvert^2}{r^2} +\frac{E_1(1,r\alpha)}{(r\alpha)^2} \lvert v \rvert^2
\Big) drd\theta.\]
\end{proposition}
\begin{proof}
Let $v \in \mathcal G$, then
\[q_\alpha(v,v)=\int_{V_\alpha} \lvert v_r\rvert^2 -\frac{1}{4} \frac{\lvert v \rvert^2}{r^2} drd\theta+ \int_{\mathbb R^+} \frac{1}{r^2} \left( \int_{-\alpha}^\alpha \lvert v_\theta\rvert^2 d\theta - r \lvert v(r,\alpha)\rvert^2-r\lvert v(r,-\alpha)\rvert^2 \right) dr.\]
We recognize in the bracket the sesquilinear form associated to the Robin Laplacian $B_{\alpha,r}$ defined in subsection \ref{interval}. As $v\in\mathcal G$, we have $\theta\mapsto v(\cdot,\theta) \in C^\infty(-\alpha,\alpha) \subset D(b_{\alpha,r})$. Then we can apply the min-max principle to obtain
\[q_\alpha(v,v) \geq \int_{V_\alpha}
\Big(\lvert v_r\rvert^2 -\frac{1}{4} \frac{\lvert v\rvert^2}{r^2} \Big)drd\theta+ \int_{\mathbb R^+}\left( \frac{E_1(\alpha,r)}{r^2} \int_{-\alpha}^\alpha \lvert v\rvert^2 d\theta \right ) dr.\]
The conclusion is due to the equality \eqref{2.1.2}.
\end{proof}

\subsubsection{Upper bound of Theorem \ref{Thprinc}.}\label{upbound}
The operator $B_{\alpha,r}$ defined in subsection \ref{interval}
will play a special role.
Recall that  $E_1(\alpha,r)<0$, and the associated normalized eigenfunction is
\[
 \widetilde \Phi_{\alpha}(r,\theta)=C_\alpha(r) \Phi_{\alpha}(r,\theta)=C_\alpha(r)\cosh\Big(\frac{m(r\alpha)\theta}{\alpha}\Big) \text{ for } \theta \in (-\alpha,\alpha).
\]
In this section, we will omit the lower indice and denote
\[
\Phi(r,\theta)\coloneqq\Phi_{\alpha}(r,\theta),\quad C(r)\coloneqq C_\alpha(r) \text{ and } \widetilde \Phi(r,\theta)\coloneqq C(r)\Phi(r,\theta).
\]

Define two orthogonal projections $\Pi$ and $P$ in $ L^2(V_\alpha)$ by
\begin{equation}
\label{4.2.3.1} 
\begin{aligned}
\Pi v(r,\theta)&:=f(r)\tilde \Phi(r,\theta), \quad f(r)\coloneqq \int_{-\alpha}^\alpha v(r,\theta) \tilde \Phi(r,\theta)d\theta, \\
Pv(r,\theta)&:=v(r,\theta)-\Pi v(r,\theta). 
\end{aligned}
\end{equation}
During the proof, the functions $v$ and $f$ will always be related by \eqref{4.2.3.1}.

\begin{proposition}{\label{Prop4.5}}
For all $v \in \mathcal G$ we have $\Pi v \in \mathcal G$, and
\[
q_\alpha(\Pi v,\Pi v)=\int_{\mathbb R^+} \lvert f'(r) \rvert^2 + \left ( K_\alpha(r) -\frac{1}{4r^2} + \frac{E_1(1,\alpha r)}{(\alpha r)^2} \right ) \lvert f(r) \rvert^2 dr,
\]
where $K_\alpha(r):=\displaystyle\int_{-\alpha}^\alpha\big\lvert\partial_r\tilde \Phi(r,\theta)\big\rvert^2 d\theta$.
\end{proposition}
\begin{proof}
Let $v \in \mathcal G$, then $f\in C^\infty_c(\mathbb R_+)$. As $\tilde \Phi$ is smooth, one has $\Pi v= f\tilde \Phi \in \mathcal G$. Moreover,
\begin{align}
\partial_r\Pi v(r,\theta)&=f'(r)\tilde \Phi(r,\theta)+f(r) \partial_r \tilde \Phi(r,\theta),
\label{4.2.3.4}\\
\partial_\theta \Pi v(r,\theta)&=f(r) \partial_\theta \tilde \Phi(r,\theta). \nonumber
\end{align}
The evaluation of $q_\alpha$ on $\Pi v$ gives
\begin{multline*}
q_\alpha(\Pi v,\Pi v)=\int_{V_\alpha}\left( \lvert \partial_r \Pi v (r,\theta) \vert^2 +\frac{\lvert \partial_\theta \Pi v(r,\theta) \rvert ^2}{r^2}-\frac{\lvert \Pi v(r,\theta)\rvert^2}{4 r^2} \right ) d\theta dr \\
-\int_{\mathbb R_+} \left(\frac{\lvert \Pi v (r,\alpha)\rvert^2}{r}+\frac{\lvert \Pi v (r,-\alpha)\rvert^2}{r}\right ) dr.
\end{multline*}
By \eqref{4.2.3.4}, 
\[
\int_{V_\alpha}\lvert \partial _r \Pi v(r,\theta)\rvert^2d\theta dr= \int_{V_\alpha}\Big( \lvert f'(r) \tilde \Phi(r,\theta)\rvert^2 + \lvert f(r)\partial_r \tilde \Phi(r,\theta)\rvert^2 + 2 f'(r) f(r) \tilde \Phi(r,\theta)\partial_r\tilde \Phi(r,\theta) \Big)d\theta dr.
\]
For all $r\in\mathbb R_+$, $\displaystyle \int_{-\alpha}^\alpha \lvert \tilde \Phi(r,\theta)\rvert^2 d\theta=1$, then 
\[
\partial_r\left(\int_{-\alpha}^\alpha\lvert\tilde \Phi(r,\theta)\rvert^2d\theta\right)=2\int_{-\alpha}^\alpha \tilde \Phi(r,\theta) \partial_r \tilde \Phi(r,\theta) d\theta=0.
\]
Applying Fubini's theorem we obtain
\begin{multline*}
q_\alpha(\Pi v,\Pi v)=\int_{\mathbb R_+} \left (\lvert f'(r)\rvert^2+K_\alpha(r) \lvert f(r)\rvert^2-\frac{\lvert f(r)\rvert^2}{4r^2} \right) dr 
\\+ \int_{\mathbb R_+} \frac{\lvert f(r)\rvert ^2}{r^2}\left \{ \int_{-\alpha}^\alpha\lvert\partial_\theta \tilde \Phi(r,\theta)\rvert^2 d\theta- r\left (\lvert\tilde \Phi(r,\alpha)\rvert^2 +\lvert \tilde \Phi(r,-\alpha)\rvert^2 \right) \right\} dr,
\end{multline*}
and the expression in the curly brackets equals $E_1(\alpha,r)$ due to the choice of $\tilde \Phi$. Then,
\[
q_\alpha(\Pi v,\Pi v)=\int_{\mathbb R^+} \left ( \rvert f'(r)\rvert^2 +\left ( K_\alpha(r)-\frac{1}{4r^2}+ \frac{E_1(\alpha,r)}{r^2} \right) \lvert f(r)\rvert^2 \right)dr.
\]
To finish the proof we use \eqref{2.1.2}.
\end{proof}
In order to obtain an upper bound for the form $q_\alpha$ we need to study the quantity $K_\alpha$.
\begin{lemma}\label{majorK}
There exists  $A>0$ such that for all $r>0$ and $\alpha \in (0,\frac{\pi}{2})$ we have
$K_\alpha(r) \leq A \alpha ^2$.
\end{lemma}
\begin{proof}
We first notice that $K_\alpha$ is continuous on $\mathbb R_+$. In addition,
\begin{align*}
K_\alpha(r)&=\int_{-\alpha}^{\alpha}\lvert C'(r)\Phi(r,\theta)\lvert^2 d\theta+\int_{-\alpha}^{\alpha} \lvert C(r)\partial_r \Phi(r,\theta)\lvert^2 d\theta+2\int_{-\alpha}^{\alpha} C'(r)C(r)\Phi(r,\theta)\partial_r \Phi(r,\theta) d\theta \\
&\leq 2 \int_{-\alpha}^{\alpha}\lvert C'(r) \Phi(r,\theta)\lvert^2 d\theta + 2 \int_{-\alpha}^{\alpha} \lvert C(r)\partial_r \Phi(r,\theta)\lvert^2 d\theta \\
&\coloneqq 2T_1+2T_2.
\end{align*}
Then,
\[
T_1=\left\lvert\frac{C'(r)}{C(r)}\right\rvert^2=\alpha^2\lvert m'(r\alpha)\rvert^2\left (\frac{\sinh\big(2m(r \alpha)\big)}{2m(r \alpha)}+1\right)^{-2}\left(\frac{\cosh\big(2m(r \alpha)\big)}{2m(r \alpha)}-\frac{\sinh\big(2m(r\alpha)\big)}{(2m(r \alpha))^2}\right).
\]
We also give an upper bound for the second term $T_2$:
\[
T_2 \leq \lvert C(r)\rvert ^2 \lvert m'(r\alpha)\vert^2 \alpha^2 \int_{-\alpha}^\alpha \sinh^2\Big(m(r\alpha)\frac{\theta}{\alpha}\Big) d\theta = \lvert C(r)\rvert ^2 \lvert m'(r\alpha)\vert^2 \alpha^3 \left ( \frac{\sinh\big(2m(r \alpha)\big)}{2m(r \alpha)}-1\right).
\]
Finally, we arrive at $K_\alpha(r) \leq 2 \alpha^2 F(r\alpha)$ with
\begin{align*}
F(r\alpha)&=\lvert m'(r\alpha)\vert^2 \left(\dfrac{\sinh\big(2m(r\alpha)\big)}{2m(r \alpha)}+1\right)^{-2} \left(\dfrac{\cosh\big(2m(r \alpha)\big)}{2m(r \alpha)}-\dfrac{\sinh\big(2m(r \alpha)\big)}{(2m(r\alpha))^2}\right)^2 \\
&\quad+\lvert m'(r\alpha) \rvert^2 \Big(\dfrac{\sinh\big(2m(r\alpha)\big)}{2m(r\alpha)}+1\Big)^{-1}
 \Big(\dfrac{\sinh\big(2m(r\alpha)\big)}{2m(r \alpha)}-1\Big).
\end{align*}
In order to conclude it is sufficient to show that $F$ is bounded on $\mathbb R_+$. As the function $F$ is continuous, we only have to prove that it admits finite limits in $0$ and $+\infty$. Let $x=r\alpha$. The function $x\mapsto m(x)$ is $C^1(\mathbb R_+)$ with $m(0)=0$,  $m(x)\sim x$ for $x\to +\infty$, and 
\[
m'(x)=-\frac{\partial_x E_1(1,x)}{2\sqrt{-E_1(1,x)}}.
\]
Furthermore, we have by \eqref{2.1.12}
\[
\partial_x E_1(1,x)=-2\,\dfrac{\cosh^2 m(x)}{\dfrac{\sinh\big(2m(x)\big)}{2m(x)}+1}.
\]
Then we can write $F(x)=G(x)+H(x)$, where
\begin{align*}
G(x)&=\frac{\cosh^4\big(m(x)\big)}{m^2(x)} \left(\frac{\sinh\big(2m(x)\big)}{2m(x)}+1\right)^{-4} \left(\frac{\cosh\big(2m(x)\big)}{2m(x)}-\frac{\sinh\big(2m(x)\big)}{(2m(x))^2}\right)^2,\\
H(x)&=\frac{\cosh^4\big(m(x)\big)}{m^2(x)} \dfrac{\dfrac{\sinh\big(2m(x)\big)}{2m(x)}-1}{\Big(\dfrac{\sinh\big(2m(x)\big)}{2m(x)}+1\Big)^3}.
\end{align*}
After a direction computation using the behavior of $m$ as $x\to 0$ and $x\to +\infty$ we get
\begin{align*}
G(x)&\to \frac{4}{9} \text{ as } x\to 0,\quad &H(x)&\to \frac{1}{12} \text{ as } x\to 0, \\
G(x)&\to 4 \text{ as } x\to+\infty,\quad & H(x)&\to \frac{1}{4} \text{ as } x\to+\infty.
\end{align*}
Then $F$ admits finite limits too, which concludes the proof.
\end{proof}

\begin{Corollaires}{\label{Cor4.9}}
For all $v\in \mathcal G$ and for all $\alpha\in (0,\frac{\pi}{2})$, 
\[q_\alpha(\Pi v,\Pi v) \leq \int_{\mathbb R_+}\left( \big\lvert f'(r)\big\rvert^2 +\Big(-\frac{1}{4r^2}+\frac{E_1(1,r\alpha)}{(r\alpha)^2}\Big) \big\lvert f(r)\big\rvert^2 \right ) dr + \alpha^2 A \lVert \Pi v\rVert^2_{L^2(V_\alpha)}.\]
\end{Corollaires}
The next proposition gives an upper bound for the Rayleigh quotients.
\begin{proposition}\label{upperbound}
There exist  $M>0$ and $\alpha_0>0$ such that $\Lambda_n(T_\alpha) \leq \mathcal E_n(\alpha) + M$
for all $n\in\mathbb N$
and all $\alpha\in(0,\alpha_0)$.
\end{proposition}
\begin{proof}
By Proposition \ref{Prop2.5} and Corollary \ref{Cor4.9}, for all $v\in \mathcal G$ there holds
\[
q_\alpha(\Pi v,\Pi v)\leq \int_{\mathbb R_+} \left( \big\lvert f'(r)\big\rvert^2 -\frac{1}{4r^2} \big\lvert f(r)\big\rvert^2 -\frac{1}{\alpha r} \big\lvert f(r)\big\rvert^2 \right ) dr 
+ \|\phi\|_\infty  \lVert f\rVert ^2_{L^2(\mathbb R^+)} + \alpha^2 A \lVert \Pi v \rVert ^2_{L^2(V_\alpha)}.
\]
Let $M=\|\phi\|_\infty+\alpha^2A$, then,
\begin{equation}
 q_\alpha(\Pi v,\Pi v)\leq \int_{\mathbb R_+} \left( \lvert f'(r)\rvert^2 -\frac{1}{4r^2} \lvert f(r)\rvert^2 -\frac{1}{\alpha r} \lvert f(r)\rvert^2 \right ) dr + M \lVert \Pi v\rVert^2_{L^2(V_\alpha)}.
\label{maj1}
\end{equation}
Let $n\in \mathbb N$ and $G\subset C^\infty_c(\mathbb R_+)$ such that $\dim G=n$. Denote
$\tilde G=\{v=g\tilde \Phi: g\in G\}$,
then $\tilde G \subset \mathcal G$ and $\dim \tilde G=n$. By \eqref{maj1} we obtain
\[\sup_{\substack{v\in \tilde G\\v\neq 0}} \dfrac{q_\alpha(v,v)}{\|v\|^2_{L^2(V_\alpha)}}
 \leq \sup_{\substack{g\in G\\g\neq 0}}\dfrac{ h_a(g,g)}{\|g\|^2_{L^2(\mathbb R_+)}} +M,\]
 and by min-max principle,
\[
\Lambda_n(T_\alpha)\leq \inf_{\substack{G\subset C^\infty_c(\mathbb R_+)\\ \dim G=n}}\sup_{\substack{v\in \tilde G\\v\neq 0}} \frac{q_\alpha(v,v)}{\|v\|^2}.
\]
Recall that $C^\infty_c(\mathbb R_+)$ is dense in $Q(H^\infty_\alpha)$
as $H^\infty_\alpha$ is the Friedrichs extension of $H_a$ (see subsection \ref{op1D}), hence,
\[
\inf_{\substack{G\subset C^\infty_c(\mathbb R_+)\\ \dim G=n}}\sup_{\substack{g\in G\\g\neq 0}}\frac{h_a(g,g)}{\|g\|^2}=\mathcal E_n(\alpha),
\]
which concludes the proof.
\end{proof}

\subsubsection{Lower bound of Theorem \ref{Thprinc}.}

Here we will still use the orthogonal projections $\Pi$ and $P$
defined in \eqref{4.2.3.1}. We recall that for $v\in \mathcal G$ we have $\Pi v= f\tilde \Phi \in \mathcal G$, $P v=v-\Pi v \in \mathcal G$, and $\lVert \Pi v\rVert_{L^2(V_\alpha)}=\lVert f\rVert_{L^2(\mathbb R_+)}$.

\begin{proposition}\label{sbb1}
For all $v \in \mathcal G$ and for all $\alpha \in (0,\frac{\pi}{2})$,
\begin{multline*}
q_\alpha(v,v)\geq\int_{\mathbb R_+}\left(\Big(1-\alpha^2\frac{r}{r+1}\Big)\lvert f'(r)\rvert^2 +\Big(\frac{E_1(1,\alpha r)}{(\alpha r)^2}-\frac{1}{4r^2}\Big) \lvert f(r)\rvert^2 \right)dr \\
+\int_{V_\alpha} \left((1-\alpha)\lvert \partial_r Pv(r,\theta)\rvert^2 + \Big(\frac{E_2(1,\alpha r)}{(\alpha r)^2}-\frac{1}{4r^2}-A\frac{r+1}{r}\Big)\lvert Pv(r,\theta)\rvert^2 \right)d\theta dr\\
-\alpha A \lVert \Pi v\rVert^2_{L^2(V_\alpha)},
\end{multline*}
with the constant $A$ from Lemma \ref{majorK}.
\end{proposition}

\begin{proof}
Let $v\in \mathcal G$, then
\begin{equation}
      \label{qqq1}
q_\alpha(v,v)=q_\alpha(\Pi v,\Pi v)+q_\alpha(Pv,Pv)+2\Re q_\alpha(\Pi v,Pv).
\end{equation}
The first term is known thanks to proposition \ref{Prop4.5}. Then we have 
\begin{multline*}
q_\alpha(Pv,Pv)=\int_{V_\alpha} \lvert \partial_r Pv(r,\theta)\rvert^2 -\frac{\lvert Pv(r,\theta)\rvert^2}{4r^2} drd\theta \\
+ \int_{\mathbb R^+} \frac{1}{r^2} \left(\int_{-\alpha}^\alpha \lvert \partial_\theta Pv(r,\theta)\rvert^2 d\theta-r\lvert Pv(r,\alpha) \rvert^2 -r\lvert Pv(r,-\alpha)\rvert^2 \right)dr.
\end{multline*}
Applying the spectral theorem to the operator $B_{\alpha,r}$ we obtain
\[
\int_{-\alpha}^\alpha \lvert \partial_\theta Pv(r,\theta)\rvert^2 d\theta-r\lvert Pv(r,\alpha) \rvert^2 -r\lvert Pv(r,-\alpha)\rvert^2 \geq E_2(\alpha,r)\lVert Pv\rVert^2_{L^2(-\alpha,\alpha)}, \quad r>0.
\]
Using the equality $E_2(\alpha,r)=\alpha^{-2}E_2(1,r\alpha)$ we finally get
\begin{equation}
\label{sbb1.2}
q_\alpha(Pv,Pv)\geq \int_{V_\alpha}\left ( \lvert\partial_r Pv(r,\theta)\rvert^2+\left(\frac{E_2(1,r\alpha)}{(r\alpha)^2}-\frac{1}{4r^2}\right)\lvert Pv(r,\theta)\rvert^2 d\theta\right) dr.
\end{equation}
To estimate the last term in \eqref{qqq1} we write
\begin{multline}
\label{croise1}
q_\alpha(\Pi v,Pv)=\int_{V_\alpha} \left (\overline{(\partial_r\Pi v)}(\partial_r Pv) -\frac{\overline{\Pi v} Pv}{4r^2} \right) d\theta dr \\
+\int_{\mathbb R^+} \Big(\frac{1}{r^2} \int_{-\alpha}^\alpha \overline{(\partial_\theta \Pi v)}(\partial_\theta Pv) d\theta-r
 \big( \overline{\Pi v}(r,\alpha)Pv(r,\alpha)+\overline{\Pi v}(r,-\alpha)Pv(r,-\alpha)\big)\Big )dr.
\end{multline}
The functions $\Pi v$ and $Pv$ are orthogonal with respect to the scalar product of $L^2(-\alpha,\alpha)$, 
\[
\int_{-\alpha}^\alpha \overline{\mathstrut\Pi v} Pvd\theta=0,
\]
and  with respect to the form $b_{\alpha,r}$, i.e. 
\[
\int_{-\alpha}^\alpha \overline{(\partial_\theta \Pi v)}( \partial_\theta Pv) d\theta-r \Big( \overline{\Pi v}(r,\alpha)Pv(r,\alpha)+\overline{\Pi v}(r,-\alpha)Pv(r,-\alpha)\Big)=0.
\]
Then \eqref{croise1} becomes
\begin{equation}
\label{croise2}
\begin{aligned}
q_\alpha(\Pi v,Pv)&=\int_{V_\alpha}\overline{(\partial_r\Pi v)}( \partial_r Pv) d\theta dr\\
&=\int_{V_\alpha}\Big(\overline{f'(r)}\tilde \Phi(r,\theta) + \overline{f(r)} \partial_r\tilde \Phi(r,\theta) \Big) \partial_r Pv(r,\theta) d\theta dr.
\end{aligned}
\end{equation}
In addition, $\displaystyle \int_{-\alpha}^\alpha Pv(r,\theta) \tilde \Phi(r,\theta) d\theta=0$, and the derivative in $r$ gives
\begin{align*}
\int_{-\alpha}^\alpha \left(\partial_r Pv(r,\theta)\right) \tilde \Phi(r,\theta) d\theta= -\int_{-\alpha}^\alpha Pv(r,\theta)  \partial_r \tilde \Phi(r,\theta) d\theta.
\end{align*}
The substitution into \eqref{croise2} gives:
\begin{equation}
\label{croise3}
q_\alpha(\Pi v,Pv)=\int_{V_\alpha}\overline{f(r)}\left(\partial_r\tilde \Phi(r,\theta)\right)\left( \partial_r Pv(r,\theta)\right) drd\theta- \int_{V_\alpha} \overline{f'(r)} \left(\partial_r\tilde \Phi(r,\theta)\right) Pv(r,\theta) drd\theta.
\end{equation}
To get a lower bound, we notice that
\[
\left \lvert\int_{V_\alpha}\overline{f(r)}\partial_r\tilde \Phi(r,\theta)\partial_r P v(r,\theta) drd\theta \right \rvert \leq \frac{1}{2\alpha} \int_{V_\alpha} \lvert f(r)\rvert^2 \lvert \partial_r \tilde \Phi(r,\theta)\rvert^2 drd\theta+\frac{\alpha}{2} \int_{V_\alpha} \lvert \partial_r Pv(r,\theta)\rvert^2drd\theta,
\]
and using lemma \ref{majorK} we obtain
\begin{equation}
\label{croise4}
\left \lvert \int_{V_\alpha}\overline{f(r)}\partial_r\tilde \Phi(r,\theta)\partial_rP v(r,\theta) drd\theta \right \rvert \leq \frac{\alpha A}{2}\int_{\mathbb R_+}\lvert f(r)\rvert^2 dr +\frac{\alpha}{2} \int_{V_\alpha}\lvert\partial_rPv(r,\theta)\rvert^2 drd\theta.
\end{equation}
Furthermore for any $\epsilon(r)>0$ we have
\begin{multline*}
\left \lvert \int_{V_\alpha} \overline{f'(r)} Pv(r,\theta)\partial_r \tilde \Phi(r,\theta)drd\theta \right \rvert  \\
 \leq \frac{1}{2} \int_{V_\alpha} \frac{1}{\epsilon(r)} \lvert Pv(r,\theta)\rvert^2 drd\theta+\frac{1}{2}\int_{V_\alpha} \epsilon(r) \lvert f'(r)\rvert^2 \lvert \partial_r \tilde \Phi(r,\theta)\rvert^2 drd\theta,
\end{multline*}
and using again Lemma \ref{majorK} and $\epsilon(r)=\displaystyle \frac{r}{A(r+1)}$ we get
\begin{multline}
\label{croise5}
\left \lvert \int_{V_\alpha} \overline{f'(r)} Pv(r,\theta)\partial_r \tilde \Phi(r,\theta)drd\theta \right \rvert \\
\leq \frac{1}{2} \int_{V_\alpha} \frac{A(r+1)}{r} \lvert Pv(r,\theta)\rvert^2 drd\theta+\frac{\alpha^2}{2}\int_{\mathbb R_+} \frac{r}{r+1} \lvert f'(r)\rvert^2 dr.
\end{multline}
The substitution of \eqref{croise4} and \eqref{croise5} into \eqref{croise3} gives us the lower bound for the cross-term. Combining it with \eqref{sbb1.2} finally gives us the following lower bound for $q_\alpha$:
\begin{multline*}
q_\alpha(v,v)\geq\int_{\mathbb R_+}\Big((1-\alpha^2\frac{r}{r+1})\lvert f'(r)\rvert^2 +\left(K_\alpha(r)+\frac{E_1(1,\alpha r)}{(\alpha r)^2}-\frac{1}{4r^2}\right) \lvert f(r)\rvert^2 \Big)dr \\
+\int_{V_\alpha} \left((1-\alpha)\lvert \partial_r Pv(r,\theta)\rvert^2 + \left(\frac{E_2(1,\alpha r)}{(\alpha r)^2}-\frac{1}{4r^2}-A\frac{r+1}{r}\right)\lvert Pv(r,\theta)\rvert^2 \right)d\theta dr\\
-\alpha A \lVert f\rVert^2_{L^2(0,+\infty)}.
\end{multline*}
We notice that $K_\alpha$ is positive, which concludes the proof.
\end{proof}
The next proposition gives us a lower bound of $q_\alpha$ in terms of the one-dimensional operator $H^\infty_a$
defined in subsection \ref{op1D}.
\begin{proposition}\label{sbb2}
There exists $K\in\mathbb R_+$ such that for all $\alpha\in (0,1)$ and for all $v \in \mathcal G$, 
\[
q_\alpha(v,v)\geq (1-\alpha^2) \int_{\mathbb R_+}
\bigg[ \lvert f'(r)\rvert^2 -\left(\frac{1}{4r^2}+\frac{1}{r\alpha(1-\alpha^2)}\right) \lvert f(r)\rvert^2 
\bigg]dr - K \lVert v \rVert ^2_{L^2(V_\alpha)}.\]
\end{proposition}
In order to prove Proposition~\ref{sbb2}, we will need some preliminary constructions. Let us introduce a new sesquilinear form,
\[
a_1(u,u)=\int_{V_\alpha}(1-\alpha) \Big(\lvert\partial_r u(r,\theta)\rvert^2 + W(r,\alpha) \lvert u(r,\theta)\rvert^2
\Big) drd\theta,\quad  u\in\mathcal G,
\]
where
\[
W(r,\alpha)=\displaystyle\frac{E_2(1,r\alpha)}{(r\alpha)^2}-\displaystyle\frac{1}{4r^2}-A\displaystyle\frac{r+1}{r}.
\]

\begin{lemma}\label{a1sbb}
There exists $C_1\in\mathbb R_+$ such that, for all $\alpha \in (0,1)$ and for all $u\in D(a_1)$, 
\[
a_1(u,u)\geq -C_1 \lVert u\rVert^2_{L^2(V_\alpha)}.
\]
\end{lemma}

\begin{proof}
It is sufficient to check that, for all $\alpha\in(0,1)$, the potential $r\mapsto W(r,\alpha)$ is
uniformly semi-bounded from below on $\mathbb R_+$. Notice that
\begin{equation}
\label{min1}
\inf_{r\in\mathbb R_+} W(r,\alpha)=\inf_{r\in\mathbb R_+} W\Big(\frac{r}{\alpha},\alpha\Big) = \inf_{x\in\mathbb R_+} \frac{E_2(1,x)}{x^2}-\frac{\alpha^2}{4x^2}-A\frac{x+\alpha}{x} \geq \inf_{x\in\mathbb R_+} h(x) -A,
\end{equation}
with
\[
h(x) \coloneqq \displaystyle\frac{E_2(1,x)}{x^2}-\displaystyle\frac{1}{4x^2}-\displaystyle\frac{A}{x}.
\]
 Clearly, $h$ is continuous on $\mathbb R_+$. In addition, by \eqref{2.1.6} we have $h(x)\to -1$ as $x\to +\infty$.
Furthermore, $E_2(1,0)=\frac{\pi^2}{4}$ by \eqref{2.1.4}. Hence, $h(x)\sim \frac{\pi^2-1}{4x^2}\to+\infty$ as $x\to0$, and
we can conclude that $h$ admits a finite lower bound on $\mathbb R_+$. Then, by \eqref{min1}, there exists $C_1\in\mathbb R_+$ such that $\inf_{r\in\mathbb R_+} W(r,\alpha) \geq -C_1$
for all $\alpha\in(0,1)$.
\end{proof}
Introduce another sesquilinear form,
\[
a_2(u,u)=\int_{\mathbb R_+}
\bigg[
 \left(1-\alpha^2 \frac{r}{r+1}\right)\lvert u'(r)\rvert^2 + \left(\frac{E_1(1,r\alpha)}{(r\alpha)^2} -\frac{1}{4r^2}\right) \lvert u(r)\rvert^2
\bigg] dr, \quad u\in C^\infty_c(\mathbb R_+).
\]
\begin{lemma}\label{a2sbb}
There exists $C_2 \in \mathbb R_+$ such that for all $\alpha \in (0,1)$ and for all $u\in D(a_2)$,
\[
a_2(u,u)\geq (1-\alpha^2) \int_{\mathbb R_+} \bigg[ \lvert u'(r)\rvert^2 +\left(-\frac{1}{4r^2}-\frac{1}{r\alpha (1-\alpha^2)} \right)\lvert u(r)\rvert^2 \bigg]dr -C_2 \lVert u\rVert^2_{L^2(\mathbb R_+)}.
\]
\end{lemma}
\begin{proof}
Using the function $\phi$ from Proposition~\ref{Prop2.5} we can write, for all $u\in D(a_2)$,
\[
a_2(u,u)\geq \int_{\mathbb R_+} \bigg[(1-\alpha^2\frac{r}{r+1}) \lvert u'(r)\rvert^2 +\left(-\frac{1}{4r^2} -\frac{1}{r\alpha}\right) \lvert u(r)\rvert^2 \bigg]dr -\|\phi\|_\infty \lVert u\rVert ^2_{L^2(\mathbb R_+)}.
\]
We want to study the right hand side of this inequality. To do that we separate the integral in two parts: a first integral on $(0,1)$ and a second one on $(1,+\infty)$.

On one hand, we notice that
\[
\int_0^1 \frac{r}{r+1} \lvert u'(r)\rvert^2 dr \leq \int_0^1 \lvert u'(r)\rvert^2 r dr.
\]
And, by an improved Hardy type inequality from \cite[Lemma A.1]{ref9}, as $u \in C^\infty_c(\mathbb R_+)$,
\begin{equation}
   \label{hardy}
\int_0^1\left( \lvert u'(r)\rvert^2 -\frac{\lvert u(r)\rvert^2}{4r^2} \right) dr \geq \int_0^1 \left ( \lvert u'(r) \rvert^2 + \frac{\lvert u(r)\rvert^2}{4r^2} \right) r dr.
\end{equation}
Then we have
\[
\int_0^1 \lvert u'(r)\rvert^2 r dr \leq \int_0^1 \bigg[
\lvert u'(r)\rvert ^2 - \left(\frac {1}{4r^2}+ \frac{1}{4r}\right)\lvert u(r)\rvert^2 \bigg]dr,
\]
and we can write
\begin{multline*}
\int_0^1 \bigg[\left(1-\alpha^2\frac{r}{r+1}\right) \lvert u'(r)\rvert^2 + \left(-\frac{1}{4r^2}-\frac{1}{r\alpha}\right)\lvert u(r)\rvert^2 \bigg]dr \\
\geq (1-\alpha^2)\int_0^1 \lvert u'(r)\rvert^2 dr+ \alpha^2 \int_0^1 \left(\frac{1}{4r^2}+\frac{1}{4r}\right)\lvert u(r)\rvert^2 dr +\int_0^1 \left(-\frac{1}{4r^2}-\frac{1}{r\alpha} \right)\lvert u(r)\rvert^2 dr.
\end{multline*}
Finally,
\begin{multline}
\label{a2.1}
\int_0^1
\bigg[\left(1-\alpha^2\frac{r}{r+1}\right) \lvert u'(r)\rvert^2 + \left(-\frac{1}{4r^2}-\frac{1}{r\alpha}\right)\lvert u(r)\rvert^2
\bigg] dr \\
\geq (1-\alpha^2)\int_0^1 \bigg[
\lvert u'(r)\rvert^2 -\left(\frac{1}{4r^2}+\frac{1}{\alpha(1-\alpha^2)r}\right) \lvert u(r)\rvert^2
\bigg]dr.
\end{multline}
On the other side, the integrals on $(1,+\infty)$ can be estimated by
\begin{multline*}
\int_1^{+\infty} 
\bigg[
\left(1-\alpha^2\frac{r}{r+1} \right)\lvert u'(r)\rvert^2 -\left(\frac{1}{4r^2} +\frac{1}{r\alpha}\right) \lvert u(r)\rvert^2
\bigg] dr \\
\geq (1-\alpha^2) \int_1^{+\infty} 
\bigg[\lvert u'(r)\rvert ^2 -\left(\frac{1}{4r^2(1-\alpha^2)} +\frac{1}{r\alpha(1-\alpha^2)} \right) \lvert u(r)\rvert^2
\bigg] dr .
\end{multline*}
Using
\[
\frac{1}{4r^2(1-\alpha^2)}=\displaystyle \frac{1}{4r^2}+ \displaystyle \frac{\alpha^2}{4r^2(1-\alpha^2)}
\]
we obtain
\begin{multline}
\label{a2.2}
\int_1^{+\infty} \bigg[ \left(1-\alpha^2\frac{r}{r+1} \right)\lvert u'(r)\rvert^2 -\left(\frac{1}{4r^2} +\frac{1}{r\alpha}\right) \lvert u(r)\rvert^2 \bigg]dr \\
\geq (1-\alpha^2) \int_1^{+\infty} \bigg[\lvert u'(r)\rvert^2 -\left(\frac{1}{4r^2}+\frac{1}{r\alpha (1-\alpha^2)}\right ) \lvert u(r)\rvert^2 \bigg]dr -\frac{\alpha^2}{4} \lVert u \rVert ^2_{L^2(\mathbb R_+)}.
\end{multline}
Putting \eqref{a2.1} and \eqref{a2.2} together we obtain the result whith $C_2=\|\phi\|_\infty+\displaystyle\frac{\alpha^2}{4}$.
\end{proof}

\begin{proof}[Proof of Proposition~\ref{sbb2}]
Let $v\in \mathcal G$, then $f\in D(a_2)$ and $Pv\in D(a_1)$, and, by Proposition \ref{sbb1},
\[
q_\alpha(v,v)\geq a_2(f,f)+a_1(Pv,Pv)-\alpha A \lVert \Pi v \rVert^2_{L^2(V_\alpha)}.
\]
We can conclude thanks to Lemmas \ref{a1sbb} and \ref{a2sbb} and by writing $K=\max(A+C_2,C_1)$.
\end{proof}

\begin{proposition}\label{lowerbound}
There exist $\alpha_0>0$ and $M_0>0$ such that, for all $\alpha\in(0,\alpha_0)$ and for all $n\in\mathbb N$,
\[
\Lambda_n(T_\alpha)\geq \mathcal E_n(\alpha)-M_0.
\]
\end{proposition}
The proof is again based on the use of the min-max principle, and we need a preliminary assertion.
\begin{lemma}\label{diag}
We define, for all $(g,\varphi)\in C^{\infty}_c(\mathbb R_+)\times\mathcal G$, the sesquilinear form
\[
h^\mathrm{diag}\Big((g,\varphi),(g,\varphi)\Big)=\int_0^{+\infty}
\Big[ \,\lvert g'(r)\rvert^2 -\left(\frac{1}{4r^2}+\frac{1}{r\alpha(1-\alpha^2)}\right) \lvert g(r)\rvert^2
\Big] dr,
\]
and let $H^\mathrm{diag}$ be the self-adjoint operator associated
with its closure in $L^2(\mathbb R_+)\times L^2(V_\alpha)$. Then
$\Lambda_n(H^\mathrm{diag})\ge \mathcal E_n\big (\alpha(1-\alpha^2)\big)$
 for all $n\in\mathbb N$.
\end{lemma}
\begin{proof}
Let $n\in \mathbb N$. By the min-max principle,
\[
\Lambda_n(H^\mathrm{diag})=\inf_{\substack{G\subset C^{\infty}_c(\mathbb R_+)\times \mathcal G \\ \dim(G)=n}} \sup_{\substack{(g,\varphi)\in G\\ (g,\varphi)\neq (0,0)}} \frac{h^\mathrm{diag}\left((g,\varphi),(g,\varphi)\right)}{\lVert g\rVert^2_{L^2(\mathbb R_+)}+\lVert \varphi \rVert^2_{L^2(\mathbb R^2)}}.\]
Moreover, thanks again to the min-max principle, we also have:
\begin{multline*}
\mathcal E_n\big(\alpha(1-\alpha^2)\big)=\inf_{\substack{G\subset C^\infty_c(\mathbb R_+) \\ \dim(G)=n}} \sup_{\substack{g\in G\\g\neq 0}}\frac{\langle g, H^\infty_{\alpha(1-\alpha^2)}g\rangle}{\lVert g\rVert^2_{L^2(\mathbb R_+)}} \\
 = \inf_{\substack{G\subset C^{\infty}_c(\mathbb R_+)\times \{0\} \\ \dim(G)=n}} \sup_{\substack{(g,0)\in G\\ g\neq 0}} \frac{h^\mathrm{diag}\left((g,0),(g,0)\right)}{\lVert g\rVert^2_{L^2(\mathbb R_+)}} \geq \Lambda_n(H^\mathrm{diag}),
\end{multline*}
where $H^\infty_{\alpha(1-\alpha^2)}$ is the Friedrichs extension of $H_{\alpha(1-\alpha^2)}$ defined in Appendix~\ref{appb}.
As $\mathcal E_n\big(\alpha(1-\alpha^2)\big)<0$ for all $n\in\mathbb N$, one has $\Lambda_n(H^\mathrm{diag}) <0$,
and
\begin{align*}
\Lambda_n(H^\mathrm{diag})&=\inf_{\substack{G\subset C^{\infty}_c(\mathbb R_+)\times \mathcal G \\ \dim(G)=n \\ h^\mathrm{diag}_{|G}<0}} \sup_{\substack{(g,\varphi)\in G\\ (g,\varphi)\neq (0,0)}} \frac{h^\mathrm{diag}\left((g,\varphi),(g,\varphi)\right)}{\lVert g\rVert^2_{L^2(\mathbb R_+)}+\lVert \varphi \rVert^2_{L^2(\mathbb R^2)}}\\
&\geq \inf_{\substack{G\subset C^{\infty}_c(\mathbb R_+)\times \mathcal G \\ \dim(G)=n \\ h^\mathrm{diag}_{|G}<0}} \sup_{\substack{(g,\varphi)\in G\\ (g,\varphi)\neq (0,0)}} \frac{h^\mathrm{diag}\left((g,0),(g,0)\right)}{\lVert g\rVert^2_{L^2(\mathbb R_+)}}.
\end{align*}
On the other hand,

\begin{align*}
\inf_{\substack{G\subset C^{\infty}_c(\mathbb R_+)\times \mathcal G \\ \dim(G)=n \\ h^\mathrm{diag}_{|G}<0}} \sup_{\substack{(g,\varphi)\in G\\ (g,\varphi)\neq (0,0)}} \frac{h^\mathrm{diag}\left((g,0),(g,0)\right)}{\lVert g\rVert^2_{L^2(\mathbb R_+)}}
= \inf_{\substack{G\subset C^{\infty}_c(\mathbb R_+) \\ \dim(G)=n \\ h^\mathrm{diag}_{|G}<0}} \sup_{\substack{g\in G\\ g\neq 0}} \frac{\langle g, H^\infty_{\alpha(1-\alpha^2)}g\rangle}{\lVert g\rVert^2_{L^2(\mathbb R_+)}}
=\mathcal E_n\left(\alpha(1-\alpha^2)\right),
\end{align*}
namely $\Lambda_n(H^\mathrm{diag})=\mathcal E_n\big(\alpha(1-\alpha^2)\big)$.
\end{proof}
We are now able to give the proof of the lower bound of the eigenvalues.
\begin{proof}[Proof of Proposition~\ref{lowerbound}]
Let $n\in\mathbb N$. By Proposition \ref{sbb2} and the min-max principle we have
\[
\Lambda_n(T_\alpha) \geq (1-\alpha^2) \inf_{\substack{G \subset \mathcal G \\ \dim(G)=n}} \sup_{\substack{v \in G \\ v \neq 0}} \frac{h^\mathrm{diag}\left((f, Pv),(f, Pv) \right)}{\lVert f\rVert ^2_{L^2(\mathbb R^+)}+ \lVert Pv\rVert^2_{L^2(U_\alpha)}}- K.
\]
Consider the map $\mathcal J:\,   \mathcal G\ni v\mapsto (f,Pv)\in C^{\infty}_c(\mathbb R_+)\times \mathcal G$,
where $f$ and $Pv$ are defined in \eqref{4.2.3.1}. Then,
\begin{align*}
\Lambda_n(T_\alpha) &\geq (1-\alpha^2) \inf_{\substack{G'\subset \mathcal J(\mathcal G) \\ \dim (G')=n }} \sup_{\substack{(g,\varphi)\in G' \\ (g,\varphi)\neq (0,0)}} \frac{h^\mathrm{diag}\left((g,\varphi),(g,\varphi)\right)}{\lVert g \rVert^2_{L^2(\mathbb R_+)}+ \lVert \varphi \rVert^2_{L^2(V_\alpha)}} -K \\
&\geq (1-\alpha^2) \inf_{\substack{G'\subset C^{\infty}_c(\mathbb R_+)\times \mathcal G\\ \dim (G')=n }} \sup_{\substack{(g,\varphi)\in G' \\ (g,\varphi)\neq (0,0)}} \frac{h^\mathrm{diag}\left((g,\varphi),(g,\varphi)\right)}{\lVert g \rVert^2_{L^2(\mathbb R_+)}+ \lVert \varphi \rVert^2_{L^2(V_\alpha)}} -K \\
&\geq (1-\alpha^2) \mathcal E_n\left(\alpha(1-\alpha^2)\right) -K,
\end{align*}
thanks to Lemma \ref{diag}. As
\[
(1-\alpha^2)\mathcal  E_n\big(\alpha(1-\alpha^2)\big)=-\displaystyle\frac{1}{\alpha^2(2n-1)^2}-\displaystyle\frac{1}{(2n-1)^2(1-\alpha^2)},
\]
we can estimate
\[
(1-\alpha^2)\mathcal E_n\big(\alpha(1-\alpha^2)\big)\geq \mathcal E_n(\alpha)-\frac{1}{1-\alpha^2}-K. \qedhere
\]
\end{proof}

Now we are able to finish the proof of Theorem \ref{Thprinc}.
Thanks to proposition \ref{upperbound} and \ref{lowerbound}, there exist
$\alpha_0\in(0,1)$, $M\in \mathbb R$ and $m\in \mathbb R$ such that, for all $\alpha\in(0,\alpha_0)$ and for all $n\in\mathbb N$
one has $\mathcal E_n(\alpha)+m\leq \Lambda_n(T_\alpha)\leq \mathcal E_n(\alpha)+M$.
Taking $\mathcal C=\max(M,\lvert m\rvert)$ we arrive at the result.

\subsection{Complete asymptotic expansion for eigenvalues}

Theorem \ref{Thprinc} and Corollary \ref{corol43} give a first order asymptotics
for the eigenvalues. In particular, it follows that each discrete
eigenvalue is simple as the angle is small. This can be used
to apply the standard perturbation theory to obtain
a full asymptotic expansion.

\begin{theorem}\label{full}
For any $n\in \mathbb N$ there exist $\lambda_{j,n}\in \mathbb R$, $j\in\mathbb N\mathop{\cup}\{0\}$,
 such that for any $N\in\mathbb N$ one has the asymptotics
\[
E_n(T_\alpha)=\dfrac{1}{\alpha^2}
\sum_{j=0}^N \lambda_{j,n} \alpha ^{2j}+ O(\alpha^{2N}) \text{ as }
\alpha\to 0,
\]
and $\lambda_{0,n}=-\dfrac{1}{(2n-1)^2}$.
\end{theorem}
\begin{proof}
Let us consider the operator $Q_\alpha$ acting on $L^2(V_\alpha)$ and defined in Section~\ref{polarcoord}:
\[
Q_\alpha v =-\partial^2_{r} v -\frac{1}{4r^2} -\frac{1}{r^2} \partial^2_{\theta} v, \quad
\pm \frac{1}{r} \partial_\theta v(r,\pm \alpha) = v(r,\pm \alpha).
\]
Using the scaling, $\theta=\alpha\eta$, $r=\alpha t$, one shows that $Q_\alpha$
is unitarily equivalent to $\alpha^{-2}L_\alpha$,
where $L_\alpha$ acts in $L^2(V_1)$ by
\[
L_\alpha v\coloneqq -\partial^2_{t} v-\frac{v}{4t^2} -\frac{1}{\alpha^2t^2} \partial^2_{\eta} v,\quad
\pm \partial_\eta v(t,\pm 1) = \alpha^2 t v(t,\pm 1),
\]
and is associated with the sesquilinear form
\[
\ell_\alpha(u,v)=\int_{V_1} \Big(
\overline{u_t} v_t -\dfrac{\overline{u} v}{4t^2}
+\dfrac{\overline{u_\eta} v_\eta}{\alpha^2 t^2} \, \Big)dt\,d\eta
-\int_{\mathbb R_+} \dfrac{\overline{u(t,1)} v(t,1)+\overline{u(t,-1)}v(t,-1)}{t}\, dt.
\]
For the eigenvalues one has $E_n(T_\alpha)=\alpha^{-2} E_n(L_\alpha)$,
$n\in\mathbb N$, and we prefer to work with $L_\alpha$ in what follows.
Remark that, in view of Theorem~\ref{SPEC}, for any $\gamma>0$ one has
\begin{multline*}
\int_{\mathbb R_+} \dfrac{\big|v(t,1)\big|^2+\big|v(t,-1)\big|^2}{t}\, dt\\
\le
\dfrac{1}{\gamma}
\int_{V_1}
\Big( |v_t|^2 -\dfrac{|v|^2}{4t^2}
+\dfrac{|v_\eta|^2}{\alpha^2 t^2} \Big)\, dt\,d\eta
+\dfrac{\gamma \alpha^2}{\sin^2\alpha}
\int_{V_1} |v|^2 dt\,d\eta,
\quad v\in D(\ell_\alpha).
\end{multline*}
In particular, there exists $b_0>0$ and $b>0$ such that
for small $\alpha$ there holds
\begin{equation}
      \label{eqtr}
\int_{\mathbb R_+} \dfrac{\big|v(t,1)\big|^2+\big|v(t,-1)\big|^2}{t}\, dt
\le b_0\Big( \ell_\alpha(v,v)+ b \|v\|^2_{L^2(V_1)}\Big),
\quad v\in D(\ell_\alpha).
\end{equation}
Introduce the following differential expressions
\[
\mathcal L_{-1} \coloneqq -\displaystyle \frac{1}{t^2} \partial ^2_{\eta},
\quad
\mathcal L_0\coloneqq -\partial^2_{t} -\displaystyle \frac{1}{4t^2},
\]
then $L_\alpha u= \mathcal L_0 u + \alpha^{-2} \mathcal L_{-1} u$.
We look for a formal approximate eigenpair $(E_\alpha, \varphi_\alpha)$ of $L_\alpha$
of the form
\[
 E_\alpha \underset{\alpha\to 0}{\sim} \sum_{j\geq 0} \lambda_j \alpha^{2j}, \quad
\varphi_\alpha \underset{\alpha\to 0}{\sim} \sum_{j\geq 0} u_j \alpha^{2j},
\]
satisfying in the sense of formal series the following eigenvalue problem
\[
L_\alpha \varphi_\alpha \underset{\alpha\to 0}{\sim} E_\alpha \varphi_\alpha, \quad
\pm \partial_{\eta} \varphi_\alpha(t,\pm 1) \underset{\alpha\to 0}{\sim} \alpha^2 t \varphi_\alpha(t,\pm 1).
\]
By collecting the terms according to the powers of $\alpha$, one arrives at
an infinite system of partial differential equations.

To determine $(\lambda_0, u_0)$ we collect the terms containing $\alpha^{-2}$, then we have to solve
\begin{align*}
\mathcal L_{-1} u_0=0, \quad
\pm \partial_\eta u_0(t,\pm 1) =0.
\end{align*}
As a consequence, $u_0$ only depends on $t$. Collecting the terms corresponding to $\alpha^0$ we get 
\begin{align}
\label{as1}
\mathcal L_0 u_0+\mathcal L_{-1} u_1 &= \lambda_0 u_0, \\
\label{as2}
\pm \partial_\eta u_1(t,\pm 1) &= t u_0(t,\pm 1).
\end{align}
Let us denote by $\mathcal L_{-1}^N$ the operator acting as $\mathcal L_{-1}$ on $L^2(V_1)$ whose domain is
\[
D(\mathcal L_{-1}^N)=\{u\in L^2 (V_1),\quad \mathcal L_{-1}^N u \in L^2(V_1),\quad \pm \partial_{\eta}u(t,\pm 1)=0\}.
\] We already know that $u_0\in\Ker(\mathcal L^N_{-1})$ and we notice that the orthogonal projections on $\Ker(\mathcal L_{-1}^N)$ and the differential expression
$\mathcal L_0$ commute. Thus we can integrate \eqref{as1} on $(-1,1)$ and obtain 
\begin{align}
\label{as3}
-2\mathcal L_0u_0-\frac{1}{t^2} \int_{-1}^1 \partial^2_{\eta} u_1 d\eta= 2 \lambda_0 u_0.
\end{align}
Using the boundary condition \eqref{as2}, the equality \eqref{as3} becomes
\[
-\partial^2_{t} u_0-\frac{1}{4t^2} u_0-\frac{1}{t} u_0=\lambda_0 u_0.
\]
Here, we recognize a one-dimensional differential operator $H^\infty_1$ defined in section~\ref{op1D}.
Hence, we are lead to choose
\[
\lambda_{0}= -\frac{1}{(2n-1)^2},
\]
where $n\in\mathbb N$ is fixed for the rest of the proof, and $u_{0}$ is the associated normalized eigenfunction, see Appendix~\ref{appb}.
We then can get an expression for $u_1$ rewritting \eqref{as1} as
\[
\mathcal L_{-1} u_1= -\frac{1}{t} u_0.
\]
Integrating it two times in $\eta$ and using the boundary condition \eqref{as2} we obtain 
\[
u_{1}(t,\eta)=\frac{u_{0}(t)}{2} t \eta^2 + C_{1} (t),
\]
where $C_{1}$ has to be determined in the next step. Notice that the function
$t\mapsto t u_{0}(t)$ belongs to $L^2(\mathbb R_+)$ as $u_{0}$ decays exponentially, 
see Appendix~\ref{appb}.

We now can give the proof of the existence of the further terms. Let $k\in\mathbb N$ and suppose that $(\lambda_{1},...,\lambda_{k-1})$ and $(u_{1},...,u_{k-1})$ are known and satisfy 
\[
u_{l}=\sum_{i=1}^l f^i_l(u_0(t),t)\eta^{2i}+ C_{l}(t)\in L^2(V_1), \quad l=1,...,k-1,
\quad C_{l}\in\Ker\big(\mathcal L_0-\displaystyle \frac{1}{t}-\lambda_0\big)^{\perp},
\]
and the functions $t\mapsto C_{l}(t)$ and $t\mapsto f_l^i(u_0(t),t)$ decay exponentially, for all $i\le l$ and $l\le k-1$.
We want to determine $(\lambda_{k},u_{k})$. We first use the equation obtained by collecting the terms in $\alpha^{2k-2}$, i.e.
\begin{align}
\label{as4}
\mathcal L_{-1} u_{k}+\mathcal L_0u_{k-1}&=\sum_{i+j=k-1} \lambda_{j} u_{i}, \\
\label{as5}
\pm\partial_\eta u_{k}(t,\pm 1)&=tu_{k-1}(t,\pm 1).
\end{align}
Using \eqref{as4} and the hypotheses we have
\begin{multline*}
\mathcal L_{-1} u_k= \sum_{\substack{i+j=k-1 \\ i\neq 0}} \lambda_j \left( \sum_{m=1}^if_i^m(u_0(t),t)\eta ^{2m} + C_i(t) \right) +\lambda_{k-1} u_0(t) \\
-\sum_{m=1}^{k-1} \left (\mathcal L_0 f_{k-1}^m(u_0(t),t) \right)\eta^{2m} -\mathcal L_0C_{k-1}(t).
\end{multline*}
We integrate it two times in $\eta$ and we use the boundary condition \eqref{as5} to cancel the term corresponding to $\eta$, then there exists $t\mapsto \tilde C_k(t)$ such that
\begin{multline}
\label{as6}
u_k(t)=-t^2\Bigg(\sum_{\substack{i+j=k-1 \\ i\neq 0}} \lambda_j \Big( \sum_{m=1}^i f_i^m(u_0(t),t)\Big) \frac{\eta^{2m+2}}{(2m+1)(2m+2)}\\
-\sum_{m=1}^{k-1}\Big(\mathcal L_0 f_{k-1}^m(u_0(t),t)\Big) \frac{\eta^{2m+2}}{(2m+1)(2m+2)} \\
+\Big(\sum_{\substack{i+j=k-1 \\ i\neq 0}} \lambda_j C_i(t) -\mathcal L_0 C_{k-1}(t) +\lambda_{k-1} u_0(t) \Big) \frac{\eta^2}{2} 
+\tilde C_{k}(t)\Bigg).
\end{multline}
We set
\begin{align*}
f_k^m(x,t)&\coloneqq -\frac{t^2}{2m(2m-1)} \left (\sum _{i=m-1}^{k-1} \lambda_{k-1-i}  f^{m-1}_i(x,t) -\mathcal L_0f_{k-1}^{m-1}(x,t)\right),\quad m=2,...,k, \\
f_k^1(x,t)&\coloneqq -\frac{t^2}{2} \left(\sum_{\substack{i+j=k-1\\i \neq 0}} \lambda_jC_i(t) -\mathcal L_0C_{k-1}(t)+\lambda_{k-1} x\right), \\
C_k(t)&\coloneqq -t^2 \tilde C_k(t).
\end{align*}
Notice that each $t\mapsto f^m_k(u_0(t),t)$, $m=1,...k$, is then in $L^2(\mathbb R_+)$ and decay exponentially due to the hypothesis on $\big(f_i^m(u_0(\cdot),\cdot)\big)_{m=1}^i$, $i=1,...,k-1$, and $\big(C_i(\cdot)\big)_{i=1}^{k-1}$. Then \eqref{as6} can be written in the form 
\[
u_k(t,\eta)= \sum_{m=1}^kf_k^m (u_0(t),t) \eta^{2m} +C_k(t).
\]
We have now to determine $\lambda_k$ and $C_k$. Let us consider the equation obtained after collecting the terms in $\alpha^{2k}$:
\begin{align}
\label{as7}
\mathcal L_{-1} u_{k+1} +\mathcal L_0 u_k&=\sum_{i+j=k} \lambda_j u_i, \\
\label{as8}
\pm \partial_\eta u_{k+1} (t,\pm 1) &= t u_k(t,\pm 1).
\end{align}
The integration of \eqref{as7} on $(-1,1)$ with respect to $\eta$ and the boundary conditions \eqref{as8} give,
\begin{multline}
\label{as9}
\left(\mathcal L_0-\frac{1}{t}-\lambda_0\right) C_k(t)= \frac{1}{t} \sum_{m=1}^k f_k^m(u_0(t),t)-\sum_{m=1}^k \frac{1}{2m+1} \mathcal L_0 f_k^m(u_0(t),t) \\
+ \sum_{\substack{i+j=k\\ i\neq 0, j\neq 0}} \lambda_j\left(\sum_{l=1}^i \frac{1}{2l+1} f_i^l(u_0(t),t) + C_i(t) \right) + \lambda_k u_0(t).
\end{multline}
This equation admits a solution $C_k \in \Ker(H^\infty_1-\lambda_0)^\perp$ iff the right hand side belongs to $\Ker (H^\infty_1-\lambda_0)^\perp$. Thus, $\lambda_k$ is uniquely determined by 
\begin{multline*}
\lambda_k=\bigg\langle -\frac{1}{t} \sum_{m=1}^k f_k^m(u_0(\cdot),\cdot)+\sum_{m=1}^k \frac{1}{2m+1} \mathcal L_0 f_k^m(u_0(\cdot),\cdot)\\
- \sum_{\substack{i+j=k\\ i\neq 0, j\neq 0}} \lambda_j\left(\sum_{l=1}^i \frac{1}{2l+1} f_i^l(u_0(\cdot),\cdot) + C_i(\cdot) \right),u_0\bigg\rangle_{L^2(\mathbb R_+)}.
\end{multline*}
As $C_k$ satisfies the inhomogeneous equation~\eqref{as9}, a standard application
of the variation of constants shows that
it is exponentially decaying, which concludes the construction of the formal asymptotics.

Now we are going to show that the above formal expression for $E_\alpha$ provides an asymptotics
for the eigenvalues of $L_\alpha$. Now let us fix $N\in \mathbb N$ and consider the finite sums
\[
E_N=\sum_{j=0}^N \lambda_j \alpha^{2j},
\quad
\varphi_N=\sum_{j=0}^N u_j \alpha^{2j}.
\]
By the preceding constructions one has
\begin{gather*}
\Big( \mathcal L_0 +\alpha^{-2} \mathcal L_{-1}\Big)\varphi_N=E_N\varphi_N +\alpha^{2N} \psi_N,
\quad
\psi_N:=\mathcal L_0 u_N-\sum_{k=N}^{2N}\Big(\sum_{i+j=k} \lambda_i u_j\Big) \alpha^{2(k-N)},\\
\pm \partial_\eta \varphi_N (t,\pm 1)= \alpha^2 t \varphi_N(t,\pm 1) - \alpha^{2N+2} t u_N(t,\pm 1).
\end{gather*}
Remark that $\varphi_N$ does not belong to the domain of $L_\alpha$
as it does not satisfy the boundary condition, but belongs to the form domain of $L_\alpha$,
and for any $v\in D(\ell_\alpha)$ one has, using the integration by parts,
\begin{multline}
     \label{eq-ma01}
\ell_\alpha(\varphi_N,v)-E_N\int_{V_1} \overline{\varphi_N} v\, dt\,d\eta\\
=\alpha^{2N}\bigg(\int_{V_1} \overline{\psi_N} v\, \,dt\,d\eta
-\int_{\mathbb R_+} \dfrac{\overline{u_N(t,1)}v(t,1)+\overline{u_N(t,-1)}v(t,-1)}{t}\, dt
\bigg).
\end{multline}

Recall that $\inf\Spec L_\alpha=\alpha^2\inf\Spec Q_\alpha=-\alpha^2(\sin^2\alpha)^{-1}$.
Furthermore, without loss of generality we may assume that the constant $b$ in \eqref{eqtr}
is such that $L'_\alpha:=L_\alpha+b\ge 1$ for small $\alpha$ and consider
the associated shifted sesquilinear form
\[
\ell'_\alpha(u,v)=\ell_\alpha(u,v)+b\langle u,v\rangle_{L^2(V_1)},
\quad
D(\ell'_\alpha)=D(\ell_\alpha).
\]
Using the preceding estimates and the Cauchy-Schwarz inequality
one deduces from \eqref{eqtr} and \eqref{eq-ma01} that for any $v\in D(\ell'_\alpha)$
there holds, with some $a_1>0$,
\begin{multline}
       \label{eq-ma02}
\big|\ell'_\alpha(\varphi_N,v)- (E_N+b) \langle \varphi_N,v\rangle_{L^2(V_1)}\big|
\\
\le a_1 \alpha^{2N}  \Big(
\ell'_\alpha(\psi_N,\psi_N)^{\frac 12}\ell'_\alpha(v,v)^{\frac 12}
+
\ell'_\alpha(u_N,u_N)^{\frac 12}\ell'_\alpha(v,v)^{\frac 12}
\Big)
\end{multline}
as $\alpha$ is sufficiently small.
Remark that for small $\alpha$ the values $\ell'_\alpha(\psi_N,\psi_N)$
and $\ell'_\alpha(u_N,u_N)$ can be estimated as $O(\alpha^{-2})$. On the other hand,
$\ell'_\alpha(\varphi_N,\varphi_N)\ge \|\varphi_N\|^2=1+O(1)$
for small $\alpha$, and it follows from \eqref{eq-ma02}
that, with some $a_2>0$, 
\[
\big|\ell'_\alpha(\varphi_N,v)- (E_N+b) \langle \varphi_N,v\rangle_{L^2(V_1)}\big|
\le a_2 \alpha^{2N-1}  
\ell'_\alpha(\varphi_N,\varphi_N)^{\frac 12}\ell'_\alpha(v,v)^{\frac 12},
\quad v\in D(\ell'_\alpha).
\]
A simple application of the spectral theorem, see Proposition~\ref{propd1} in Appendix~\ref{appd}, shows that
\[
\dist \big( \Spec L_\alpha, E_N)\le \dfrac{a_2\alpha^{2N-1}}{1-a_2\alpha^{2N-1}} (E_N+b)
=O(\alpha^{2N-1}) \text{ as $\alpha$ tends to $0$.}
\]
By Theorem~\ref{Thprinc}, the only point of the spectrum of $L_\alpha$
which can satisfy the above estimate is the $n$th eigenvalue $E_n(L_\alpha)$.
As $N$ is arbitrary, the result follows.
\end{proof}

\section{Decay of eigenfunctions}\label{sec50}

Let $\alpha\in(0,\frac{\pi}{2})$ be fixed. The following proposition (Agmon-type estimate)
shows that the eigenfunctions of $T_\alpha$ corresponding to the discrete eigenvalues
are localized near the vertex of the sector.

\begin{theorem}\label{agmon}
Let $E$ be a discrete eigenvalue of $T_\alpha$ and $\mathcal V$
be an associated  eigenfunction, then for any $\epsilon\in(0,1)$
one has
\begin{equation}
      \label{agm2}
\int_{U_\alpha} \big( \lvert \nabla \mathcal V\rvert^2 + \lvert \mathcal V\rvert^2 \big) e^{2(1-\epsilon) \sqrt{-1-E} \lvert x \rvert}dx < +\infty.
\end{equation}
\end{theorem}

We remark that the term $(-1-E)$ appearing in the exponential is exactly the distance between
the eigenvalue and the bottom of the essential spectrum.

\begin{proof}
Let $\epsilon\in(0,1)$ and $L>0$. Define
$\phi_L(x)=\sqrt{-1-E} \min\big(\lvert x\rvert,L\big)$.
We are going to show first that there exists $K_\epsilon>0$ such that
\begin{equation}
   \label{agm1}
\lVert \mathcal V e^{(1-\epsilon)\phi_L} \rVert^2_{H^1(U_\alpha)} \leq K_\epsilon.
\end{equation}
The proof follows essentially the same steps as \cite[Proposition 2.8]{ref2} showing a similar result
for the lowest eigenvalue of corner domains.

Let $\chi_0$ and $\chi_1$ be smooth functions of $\mathbb R_+$ satisfying
\[
\chi_0(t)=1 \text{ for } 0<t<1, \quad
\chi_0(t)=0 \text{ for } t>2, \quad
\chi_0^2(t)+\chi_1^2(t)=1.
\]
For $R>0$, consider the functions
$\chi_{j,R}(x)=\chi_j\big(\lvert x \rvert/R\big)$
 defined on $U_\alpha$, $j=0,1$.
We get easily, for $u\in H^1(U_\alpha)$ and for $\gamma >0$,
\[
t^\gamma_\alpha(u,u)=t^\gamma_\alpha(u\chi_{0,R},u\chi_{0,R})+t^\gamma_\alpha(u\chi_{1,R},u\chi_{1,R}) -\sum_{j=0,1} \lVert u\nabla \chi_{j,R} \lVert^2_{L^2(U_\alpha)}.
\]
In particular, there exists $C>0$ such that 
\begin{align}
\label{EA}
t^\gamma_\alpha(u,u)\geq t^\gamma_\alpha(u\chi_{0,R},u\chi_{0,R})+t^\gamma_\alpha(u\chi_{1,R},u\chi_{1,R}) -\frac{C}{R^2} \lVert u \rVert^2_{L^2(U_\alpha)},
\quad
u\in H^1(U_\alpha), \quad R>0.
\end{align}
For $\delta\in (0,1)$ we have $t_\alpha(u,u)=\delta \lVert \nabla u \rVert_{L^2(U_\alpha)}^2 +(1-\delta) t^{\frac{1}{1-\delta}}_\alpha(u,u)$,
and \eqref{EA} leads to
\begin{multline}
\label{EA1}
t_\alpha(u,u)\geq \delta\lVert \nabla u \rVert ^2_{L^2(U_\alpha)} \\
+(1-\delta) \left(t^{\frac{1}{1-\delta}}_\alpha(u\chi_{0,R},u\chi_{0,R})+t^{\frac{1}{1-\delta}}_\alpha(u\chi_{1,R},u\chi_{1,R}) -\frac{C}{R^2} \lVert u \rVert^2_{L^2(U_\alpha)}\right).
\end{multline}
We are going to provide a lower bound for the first two terms in the bracket.
As $u\chi_{0,R}\in H^1(U_\alpha)$, we have immediatly by the min-max principle
\begin{align}
\label{EA2}
t_\alpha^{\frac{1}{1-\delta}}(u\chi_{0,R},u\chi_{0,R})\geq E_1(T^{\frac{1}{1-\delta}}_\alpha) \lVert u\chi_{0,R}\rVert^2_{L^2(U_\alpha)}=
-\dfrac{1}{(1-\delta)^2\sin^2\alpha}
\lVert u\chi_{0,R}\rVert^2_{L^2(U_\alpha)}.
\end{align}
To estimate $t^{\frac{1}{1-\delta}}_\alpha(u\chi_{1,R},u\chi_{1,R})$ we introduce the domain
$U_\alpha^R=\{x\in U_\alpha, \text{ } \lvert x \rvert \geq R \}$,
and let $C_R$ be the sector obtained by translation of vector $(R,0)$ of $U_\alpha$.
Define
\[
D^+=\big(U_\alpha^R\backslash \overline{C_R}\big)\cap \big(\mathbb R \times \mathbb R_+\big), \quad
D^-=\big(U_\alpha^R\backslash \overline{C_R}\big)\cap \big(\mathbb R \times \mathbb R_-\big),
\]
see Figure~\ref{dessinEA2}, then
$\overline{U_\alpha^R}=\overline{D^+\cup D^- \cup C_R}$.

\begin{figure}
	\centering
		\includegraphics[width=0.40\textwidth]{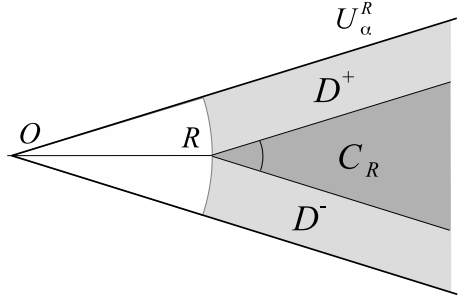}
	\caption{\label{dessinEA2}Partition of $U^R_\alpha$.}
\end{figure}

Consider the  sesquilinear forms
\begin{gather*}
q^\pm(u,u)=\int_{D^\pm}\lvert \nabla u\rvert ^2 dx -\frac{1}{1-\delta}\int_{\partial U_\alpha\cap \partial D^\pm}\lvert u\rvert^2 ds, 
\quad
D(q^\pm)=\big\{u\in H^1(D^\pm),\text{ } u=0 \text{ if } \lvert x \rvert=R\big\},\\
q_{C_R}(u,u)=\int_{C_R}\lvert \nabla u\rvert^2 dx, \quad D(q_{C_R})=H^1(C_R),
\end{gather*}
and denote by $Q^\pm$ and $Q_{C_R}$ the associated self-adjoint operators
in $L^2(D^\pm)$ and $L^2(C_R)$ respectively.
As $Q^+$ and $Q^-$ are unitarily equivalent and $q_{C_R}$ is non-negative, we have
\begin{multline*}
t_\alpha^{\frac{1}{1-\delta}}(u\chi_{1,R},u\chi_{1,R})= q^+(u\chi_{1,R},u\chi_{1,R})+q^-(u\chi_{1,R},u\chi_{1,R})+q_{C_R}(u\chi_{1,R},u\chi_{1,R}) \\
\geq \Lambda_1(Q^+)\left ( \lVert u\chi_{1,R}\rVert^2_{L^2(D^+)}+\lVert u\chi_{1,R}\rVert^2_{L^2(D^-)}\right),
\end{multline*}
We define 
\[
\mathcal R_\alpha=\Big\{x\in\mathbb R^2: \frac{x_2}{\tan\alpha}\leq x_1\leq \frac{x_2}{\tan\alpha}+R\Big\}.
\]
Remark that, if one takes $u\in D(q^+)$ and denotes by $\widetilde u$ its extension by zero to $\mathcal R_\alpha$ then, $q^+(u,u)=q_{\mathcal R_\alpha}(\widetilde u,\widetilde u)$, where
\[
q_{\mathcal R_\alpha}(u,u)=\int_{\mathcal R_\alpha}\lvert \nabla u\rvert ^2 dx -\frac{1}{1-\delta} \int_{\mathbb R} \Big\lvert u \Big(\frac{x_2}{\tan\alpha},x_2\Big)
\Big\rvert^2  \frac{dx_2}{\sin \alpha}, \quad
u\in H^1(\mathcal R_\alpha).
\]
Then $\inf\Spec Q^+ \geq \inf\Spec Q_{\mathcal R_\alpha}$, where $Q_{\mathcal R_\alpha}$ is the self-adjoint operator associated to $q_{\mathcal R_\alpha}$ acting in $L^2(\mathcal R_\alpha)$.
By applying an anti-clockwise rotation of angle $\frac{\pi}{2}-\alpha$
one sees that $Q_{\mathcal R_\alpha}$ is unitarily equivalent to $T_{RN}\otimes 1 + 1\otimes L$, where
$T_{RN}$ is the (Robin-Neumann) Laplacian acting in $L^2(0,R\sin \alpha)$ and defined on
\[
D(T_{RN})=\Big\{u\in H^2(0,R\sin \alpha), \text{ } -u'(0)=\frac{1}{1-\delta} u(0) \text{ and } u'(R\sin\alpha)=0\Big\},
\]
and $L$ is the free Laplacian in $L^2(\mathbb R)$. Then, $\inf\Spec Q_{\mathcal R_\alpha}=\inf\Spec T_{RN}$. By \cite[Lemma A.1]{ref2} we have
\[
\Lambda_1(T_{RN})=-\frac{1}{(1-\delta)^2} -4 e^{-2\frac{R\sin \alpha}{1-\delta}} +O(Re^{-4\frac{R\sin \alpha}{1-\delta}}), \text { } R \to \infty,
\]
and there exist $R_0>0$ and $C_0>0$ such that for all $R\geq R_0$ we have
\[
\Lambda_1(T_{RN})\geq -\frac{1}{(1-\delta)^2} -4C_0e^{-2\frac{R\sin \alpha}{1-\delta}}.
\]
Finally, for all $u\in H^1(U_\alpha)$ we have
\begin{align}
\label{EA3}
t^{\frac{1}{1-\delta}}_\alpha(u\chi_{1,R},u\chi_{1,R}) \geq \Big({}-\frac{1}{(1-\delta)^2}-4C_0e^{-2\frac{R\sin \alpha}{(1-\delta)}}\Big)\lVert u\chi_{1,R}\rVert ^2_{L^2(U_\alpha)}, \text{ } R\geq R_0.
\end{align}
Combining \eqref{EA1} with \eqref{EA2} and \eqref{EA3} we get, for all $u\in H^1(U_\alpha)$ and $R\geq R_0$,
\begin{multline}
\label{EA4}
t_\alpha(u,u)\geq \delta\lVert \nabla u \rVert ^2_{L^2(U_\alpha)} -\frac{1}{(1-\delta)\sin^2\alpha}\lVert u \chi_{0,R}\rVert^2_{L^2(U_\alpha)} \\
-\left(\frac{1}{1-\delta}+4(1-\delta)C_0e^{-\frac{2R\sin \alpha}{1-\delta}}\right) \lVert u\chi_{1,R}\rVert^2_{L^2(U_\alpha)} -\frac{C(1-\delta)}{R^2} \lVert u\rVert^2_{L^2(U_\alpha)}.
\end{multline}

Denote $\psi_{L,\epsilon}=(1-\epsilon)\phi_L$. The functions $\psi_{L,\epsilon}$ and $\nabla \psi_{L,\epsilon}$ belong to $L^\infty(\mathbb R^2)$ and
\begin{align}
\label{EA5}
\big\lVert \nabla \psi_{L,\epsilon}\big\rVert^2_\infty \leq (1-\epsilon)^2 ({}-1-E).
\end{align}
An integration by parts, see \cite[Lemma 2.7]{ref2} for details, gives
\[
t_\alpha( \mathcal V e^{\psi_{L,\epsilon}},\mathcal V e^{\psi_{L,\epsilon}}) =\int_{U_\alpha} \lvert \mathcal V e^{\psi_{L,\epsilon}}\rvert^2
\Big(E+\lvert \nabla \psi_{L,\epsilon} \rvert^2\Big) dx.
\]
As $\mathcal V e^{\psi_{L,\epsilon}} \in H^1(U_\alpha)$, we obtain by \eqref{EA4}, for all $R\geq R_0$,
\begin{multline*}
\int_{U_\alpha} \lvert \mathcal V e^{\psi_{L,\epsilon}}\rvert^2 \left (E+\lvert \nabla \psi_{L,\epsilon} \rvert^2\right ) dx \\
\geq \delta\lVert \nabla (\mathcal V e^{\psi_{L,\epsilon}}) \rVert ^2_{L^2(U_\alpha)} -\frac{1}{(1-\delta)\sin^2\alpha}\lVert \mathcal V e^{\psi_{L,\epsilon}} \chi_{0,R}\rVert^2_{L^2(U_\alpha)} \\
-\left(\frac{1}{1-\delta}+4(1-\delta)C_0e^{-\frac{2R\sin \alpha}{1-\delta}}\right) \lVert \mathcal V e^{\psi_{L,\epsilon}}\chi_{1,R}\rVert^2_{L^2(U_\alpha)} -\frac{C(1-\delta)}{R^2} \lVert \mathcal V e^{\psi_{L,\epsilon}}\rVert^2_{L^2(U_\alpha)},
\end{multline*}
which can be transformed in virtue of~\eqref{EA5} into
\begin{align}
\label{EA6}
A_0 \lVert \mathcal V e^{\psi_{L,\epsilon}}\chi_{0,R}\rVert ^2_{L^2(U_\alpha)} \geq \delta \lVert\nabla(\mathcal V e^{\psi_{L,\epsilon}}\rVert^2_{L^2(U_\alpha)} + A_1 \lVert \mathcal V e^{\psi_{L,\epsilon}} \chi_{1,R} \rVert^2_{L^2(U_\alpha)},\text{ } R\geq R_0,
\end{align}
with
\begin{align*}
A_0&=(\epsilon^2-2\epsilon) (-1-E) +\frac{\cos^2\alpha+\delta\sin^2\alpha}{(1-\delta)\sin^2\alpha}+\frac{C(1-\delta)}{R^2},\\
A_1&=(2\epsilon-\epsilon^2)(-1-E) -\frac{\delta}{1-\delta}-4(1-\delta)C_0e^{-\frac{2R\sin\alpha}{1-\delta}}-\frac{C(1-\delta)}{R^2}.
\end{align*}
As $\epsilon \in (0,1)$, we have $0<2\epsilon-\epsilon^2<1$.
In addition, $0<-1-E\le -1-E_1(T_\alpha)=\cot^2\alpha$.
Therefore, one can find $R_\epsilon >R_0$ and $\delta_\epsilon \in (0,1)$ such $A_0>0$ and $A_1>0$
that for all $R\geq R_\epsilon$. Furthermore, there exists $m_\epsilon>0$ such that $A_1\geq m_\epsilon$ for $R\geq R_\epsilon$. For the same $\delta_\epsilon$ we can show that there exits $M_\epsilon>0$ such that $A_0\leq M_\epsilon$ for all $R\geq R_\epsilon$. 
The inequality \eqref{EA6} implies then
\[
M_\epsilon\lVert\mathcal V e^{\psi_{L,\epsilon}} \chi_{0,R}\rVert^2_{L^2(U_\alpha)} \geq \delta_\epsilon \lVert \nabla (\mathcal V e^{\psi_{L,\epsilon}})\rVert^2_{L^2(U_\alpha)} +m_\epsilon \lVert \mathcal V e^{\psi_{L,\epsilon}} \chi_{1,R}\rVert^2_{L^2(U_\alpha)},\quad R\geq R_\epsilon,
\]
and, finally,
\[
C_\epsilon\lVert \mathcal V e^{\psi_{L,\epsilon}}\chi_{0,R} \rVert^2_{L^2(U_\alpha)} \geq \lVert \mathcal V e^{\psi_{L,\epsilon}}\rVert^2_{H^1(U_\alpha)}, \text{ }R\geq R_\epsilon,
\]
with $C_\epsilon=\frac{M_\epsilon}{\delta_\epsilon+m_\epsilon}+1$. Notice that, 
\[
\lVert\mathcal V e^{\psi_{L,\epsilon}}\chi_{0,R}\rVert^2_{L^2(U_\alpha)} \leq e^{4(1-\epsilon)\sqrt{-1-E}R} \lVert\mathcal V \chi_{0,R}\rVert^2_{L^2(U_\alpha)}\leq e^{4(1-\epsilon)\sqrt{-1-E}R}\lVert \mathcal V \rVert^2_{L^2(U_\alpha)},
\]
which gives \eqref{agm1}.

Now let us pass from \eqref{agm1} to \eqref{agm2}. We have
\[
\lVert \mathcal V e^{\psi_{L,\epsilon}} \rVert^2_{H^1(U_\alpha)}=\int_{U_\alpha} \Big\lvert \nabla \mathcal V+ \mathcal V \nabla \psi_{L,\epsilon} \Big\rvert^2 e^{2\psi_{L,\epsilon}} dx + \int_{U_\epsilon} \lvert \mathcal V \rvert^2 e^{2\psi_{L,\epsilon}}dx.
\]
Using the inequality
\[
2\Big \lvert \int_{U_\alpha} e^{2\psi_{L,\epsilon}} \mathcal V \left ( \nabla  \mathcal V\cdot\nabla\psi_{L,\epsilon}\right) dx \Big \rvert\leq \sigma \lVert  \nabla \mathcal V  e^{\psi_{L,\epsilon}} \rVert ^2_{L^2(U_\alpha)} +\frac{1}{\sigma} \lVert \nabla \psi_{L,\epsilon}\mathcal V   e^{\psi_{L,\epsilon}} \rVert^2_{L^2(U_\alpha)}, \quad \sigma>0,
\]
with $\sigma=\displaystyle \frac{1}{2}$ we get
\[
\lVert \mathcal V e^{\psi_{L,\epsilon}} \rVert^2_{H^1(U_\alpha)} \geq \int_{U_\alpha} \Big (1-\lvert \nabla \psi_{L,\epsilon}\rvert^2 \Big) \lvert \mathcal V\rvert ^2 e^{2\psi_{L,\epsilon}} dx +\frac{1}{2} \int_{U_\alpha} \lvert \nabla \mathcal V\rvert^2 e^{2\psi_{L,\epsilon}} dx,
\]
and then
\[
\lVert \mathcal V e^{\psi_{L,\epsilon}} \rVert^2_{H^1(U_\alpha)} + \lVert \nabla \psi_{L,\epsilon} \mathcal V e^{\psi_{L,\epsilon}} \rVert^2_{L^2(U_\alpha)} \geq \frac{1}{2} \int_{U_\alpha} \Big ( \lvert \nabla \mathcal V \rvert^2+ \lvert \mathcal V\rvert^2\Big) e^{2\psi_{L,\epsilon}} dx.
\]
By combining \eqref{agm1} with \eqref{EA5} we arrive at 
\[
\int_{U_\alpha} \Big( \lvert \nabla \mathcal V \rvert^2+ \lvert \mathcal V\rvert^2\Big) e^{2\psi_{L,\epsilon}} dx \leq 2K_\epsilon \Big (1+(1-\epsilon )^2(-1-E) \Big).
\]
As the right-hand side does not depend on $L$, 
one can pass to the limit as $L\to+\infty$ using the monotone convergence,
which gives the result.
\end{proof}

\appendix

\section{Proof of Lemma~\ref{density}}\label{appa}
By standard arguments, $C^\infty_c(\mathbb R^2)$ is dense in $H^1(U_\alpha)$.
In order to prove Lemma~\ref{density} it is then sufficient to show that 
any function $v\in C^\infty_c(\mathbb R^2)$ can be approximated
by the functions from $\mathcal F$ in the norm of $H^1(U_\alpha)$.

Let $v\in C^\infty_c(\mathbb R^2)$. Pick a smooth function $\psi\in C^\infty(\mathbb R_+)$ with
\[
\psi(s)=1 \quad \text{if} \quad s \leq \frac{1}{2}, \quad
\psi(s)=0 \quad \text{if} \quad s \geq 1, \quad
0\le \psi\le 1,
\]
and set, for small $\epsilon >0$, 
\[
\chi_\epsilon (x)=\psi\left(\left \lvert\frac{\ln\lvert x\rvert}{\ln \epsilon}\right\rvert\right).
\]
Finally, set $v_\epsilon(x)=\chi_\epsilon(x) v(x)$. 

Notice that $\chi_\epsilon$ is radial and then there exists $\varphi_\epsilon \in C^\infty_c(\mathbb R_+)$ sucht that $\chi_\epsilon(x)=\varphi_\epsilon(\lvert x\rvert)$. In addition, $\varphi_\epsilon$ satisfies
\[
\varphi_\epsilon(r)=1 \quad \text{if} \quad r \in \Big(\sqrt{\epsilon},\frac{1}{\sqrt{\epsilon}}\Big), \quad
\varphi_\epsilon(r)=0 \quad \text{if} \quad r \leq \epsilon \quad\text{or}\quad r\geq \frac{1}{\epsilon},
\]
and ${v_\epsilon} \in \mathcal F$.
Let us show that $\lVert v-v_\epsilon\rVert_{H^1(U_\alpha)} \to 0$ as $\epsilon \to 0$.

Denote $u(r,\theta)\coloneqq v(r\cos\theta,r\sin \theta)$ and $u_\epsilon(r,\theta)\coloneqq v_\epsilon(r\cos\theta,r\sin\theta)$, then
\begin{multline*}
\lVert v - v_\epsilon \rVert^2_{L^2(U_\alpha)}= \int_{-\alpha} ^\alpha \int_0 ^{+ \infty} \lvert u(r, \theta) -u_\epsilon (r,\theta) \rvert ^2 r dr d\theta \\
= \int_{-\alpha} ^\alpha \int_0 ^{\sqrt{\epsilon}} \lvert u \vert ^2 \lvert 1 - \varphi_\epsilon \rvert ^2 r drd\theta + \int_{-\alpha}^\alpha \int_{\frac{1}{\sqrt{\epsilon}}}^{+\infty} \lvert u\rvert ^2\lvert 1 - \varphi_\epsilon \rvert ^2 r dr d\theta,
\end{multline*}
and the right-hand side tends to $0$ as $\epsilon$ is small due to the dominated convergence.
Furthermore,
\begin{multline*}
 \lVert \nabla v - \nabla v_\epsilon \rVert ^2_{L^2(U_\alpha)}  \leq 2 \int_{-\alpha} ^\alpha \int_0 ^{+ \infty} \lvert \nabla v(r\cos\theta,r\sin\theta) \rvert ^2 \lvert 1 - \varphi_\epsilon(r) \rvert ^2 r dr d\theta \\
 + 2 \int_{-\alpha} ^\alpha \int_0 ^{+ \infty } \lvert v(r\cos\theta,r\sin\theta) \vert ^2 \lvert \nabla \chi_\epsilon(r\cos\theta,r\sin\theta) \vert ^2 r dr d\theta,
\end{multline*}
and the first term tends to 0 by the dominated convergence. On the other hand,
\[ \lvert \nabla \chi_\epsilon(r\cos\theta,r\sin\theta) \vert ^2 =\lvert \varphi'_\epsilon (r)\rvert^2= \frac{1}{r^2 \lvert \ln \epsilon \rvert ^2 } \left \lvert \psi' \left ( \left \lvert \frac{\ln r}{\ln \epsilon} \right \rvert \right )\right \lvert ^2, \]
and
\begin{multline*}
\int_{-\alpha} ^\alpha \int_0 ^{+ \infty } \lvert v(r\cos\theta,r\sin\theta) \vert ^2 \lvert \nabla \chi_\epsilon(r\cos\theta,r\sin\theta) \vert ^2 r dr d\theta = \\
\int_{-\alpha}^{\alpha} \int_\epsilon^{\sqrt{\epsilon}} \lvert u(r,\theta) \rvert ^2 \lvert \varphi'_\epsilon(r) \rvert ^2 r dr d\theta + \int_{-\alpha}^{\alpha} \int_{\frac{1}{\sqrt{\epsilon}}}^{\frac{1}{\epsilon}} \lvert u(r,\theta) \rvert ^2 \lvert\varphi'_\epsilon(r) \rvert ^2 r dr d\theta.
\end{multline*}
The functions $u$ and $\psi'$ are bounded, and we can get the following upper bound:
\[
\int_{-\alpha} ^\alpha \int_0 ^{+ \infty } \big\lvert v(r\cos\theta,r\sin\theta) \big\vert^2
\big \lvert \nabla \chi_\epsilon(r\cos\theta,r\sin\theta) \big\vert^2 r dr d\theta
\leq \alpha\frac{\lVert u \rVert^2_\infty \lVert \chi' \rVert^2_\infty}{\lvert \ln\epsilon \rvert} \xrightarrow[\epsilon \to 0]{} 0,
\]
which concludes the proof.

\section{Proof of Theorem~\ref{SPEC}}\label{appa2}

\subsection{Closedness and semiboundedness.} Let $\gamma>0$, $\alpha\in (0,\pi)$ and $u\in\mathcal F$. Using \eqref{2.1.1} we have
\begin{equation}
   \label{trace1}
\begin{aligned}
t_\alpha^\gamma (u,u)&= \int_{x_2 \in \mathbb R}\int_{x_1 > \frac{\lvert x_2 \rvert }{\tan\alpha}}\lvert \nabla u(x_1,x_2)\rvert ^2 dx -\gamma \int_{\mathbb R} \lvert u \left(\frac{\lvert x_2 \rvert}{\tan\alpha},x_2 \right)\rvert ^2 \frac{dx_2}{\sin\alpha} \\
        &\geq \int_{\mathbb R} \left ( \int_{x_1 > \frac{\lvert x_2 \rvert}{\tan\alpha}} \Big\lvert \frac{\partial u}{\partial x_1}(x_1,x_2) \Big\rvert ^2 dx_1
				- \frac{\gamma}{\sin\alpha} \Big\lvert u \Big(\frac{\lvert x_2 \rvert}{\tan\alpha},x_2 \Big)\Big\rvert ^2 \right ) dx_2,\\
				&\geq -\frac{\gamma^2}{\sin^2\alpha}\lVert u\rVert^2_{L^2(U_\alpha)}.
\end{aligned}
\end{equation}
Writing $\gamma=1/\epsilon$ we conclude that for all $u\in \mathcal F$ and for all $\epsilon>0$ we have
\begin{align*}
\int_{\partial U_\alpha} \lvert u\rvert^2 ds \leq \epsilon \int_{U_\alpha}\lvert \nabla u\rvert^2 dx +\frac{1}{\epsilon\sin^2\alpha} \int_{U_\alpha}\lvert u\rvert^2 dx,
\end{align*}
which means that the trace can be extended to a bounded linear map from $H^1(U_\alpha)$
to $L^2(\partial U_\alpha)$.
Furthermore, the boundary term in $t^\gamma_\alpha$ is then infinitesimally small
with respect to the gradient term, hence,
$t^\gamma_\alpha$ is closed on $H^1(U_\alpha)$
due to the KLMN theorem, see \cite[Theorem X.17]{rs2}.

\subsection{Bottom of the spectrum for $\alpha<\frac{\pi}{2}$.}
Let us show that $\inf\Spec T^\gamma_\alpha=-\gamma^2 (\sin^2\alpha)^{-1}$
for $\alpha<\frac{\pi}{2}$. By \eqref{trace1} we have the inequality
$T^\gamma_\alpha \geq -\gamma^2 (\sin^2\alpha)^{-1}$. On the other hand, by the explicit
computation, $t^\gamma_\alpha(u_0,u_0)=-\gamma^2(\sin^2\alpha)^{-1} \|u_0\|^2_{L^2(U_\alpha)}$
for $u_0(x_1,x_2)=e^{-\gamma x_1/\sin \alpha}$. It follows that the lower bound is optimal
and that the bottom of the spectrum is an eigenvalue, and $u_0$ is an associated eigenfunction,
which proves the point (b) of the theorem.

\subsection{Lower bound for $\alpha\ge \frac{\pi}{2}$.}\label{b3}
Let us show that $\inf\Spec T^\gamma_\alpha \geq -\gamma^2$ for $\alpha \in [\frac{\pi}{2},\pi)$.
Decompose $U_\alpha$ into the following three pieces:
\begin{gather*}
D_1=\left\{x\in\mathbb R^2\text{ }: \text{ }\lvert \arg(x_1+i x_2)\rvert< \alpha-\displaystyle \frac{\pi}{2} \right\},\\
D_2=\left(U_\alpha \backslash \overline{D_1}\right ) \cap \left(\mathbb R\times\mathbb R_+\right),
\quad
D_3=\left(U_\alpha \backslash \overline{D_1}\right ) \cap\left( \mathbb R\times \mathbb R_-\right),
\end{gather*}
then $\overline{U_\alpha}=\overline{D_1\cup D_2\cup D_3}$ and we define, for $j=1,2,3$,
\[
q_j(u,u)=\int_{D_j}\lvert \nabla u\rvert^2 dx -\int_{\partial U_\alpha \cap \partial D_j} \lvert u\rvert^2 ds, \quad u\in H^1(D_j).
\]
By the min-max principle and the inclusion $H^1(U_\alpha)\subset H^1( D_1 \mathop{\cup}D_2\mathop{\cup}D_3)$ we have
\begin{equation}
 \label{infspec1}
\inf\Spec  T^\gamma_\alpha\geq \inf\Spec  Q,
\end{equation}
where $Q$ is the self-adjoint operator acting on $L^2(U_\alpha)$ associated to the sesquilinear form defined for $u\in H^1( D_1 \mathop{\cup}D_2\mathop{\cup}D_3)$ by
$q(u,u)=q_1(u,u)+q_2(u,u)+q_3(u,u)$.
We have then $Q=Q_1\oplus Q_2 \oplus Q_3$,
where $Q_j$ are the self-adjoint operator associated with $q_j$ and acting in $L^2(D_j)$.
Notice that $Q_1$ is positive and $Q_2$ and $Q_3$ are unitarily equivalent
and have the same spectrum.
Furthermore, $Q_2$ is unitarily equivalent to $T_N\otimes 1+1\otimes B_\gamma$, where $T_N$ is the Neumann Laplacian in $L^2(\mathbb R_+)$
and $B_\gamma$ is defined in subsection \ref{halfline}, which gives $\inf\Spec Q_2=-\gamma^2$.
Therefore, $\inf\Spec Q=\min\{\inf\Spec Q_1,\inf\Spec Q_1,\inf\Spec Q_1\}=-\gamma^2$,
and \eqref{infspec1} gives the result.

\subsection{Lower bound for the essential spectrum as $\alpha<\frac{\pi}{2}$.}\label{b4}

Let us show the lower bound
\begin{align}
\label{Specess0}
\inf\Specess T_\alpha^\gamma \geq -\gamma^2.
\end{align}
Let $A=(a,0)$ with $a>0$. We denote by $C_A$ the sector obtained after a translation of $U_\alpha$ along the vector $OA$. Let $H_A$ be the orhogonal projection of $A$ on the half-line
$\mathbb R_+ (1,\tan\alpha)$ and $L:=\lvert AH_A\rvert\equiv a \sin\alpha$, and, in particular, $L\to +\infty$ for $a\to +\infty$.
In the same way, we define $H'_A$ the orthogonal projection of $A$ on $\mathbb R_+(1,-\tan\alpha)$. Consider the following four domains:
\begin{align*}
D_1&=OH_AAH'_A, & D_2&=\left( U_\alpha \backslash(\overline{C_A\cup D_1})\right) \cap \left(\mathbb R\times\mathbb R_+\right),\\
D_3&=\left( U_\alpha \backslash(\overline{C_A\cup D_1})\right) \cap \left(\mathbb R\times \mathbb R_-\right), &
D_4&=C_A,
\end{align*}
see Fig~\ref{dessinSpecess}. Clearly,
$\overline{U_\alpha}=\overline{D_1\cup D_2 \cup D_3\cup D_4 }$.
\begin{figure}
  \centering
		\includegraphics[width=0.40\textwidth]{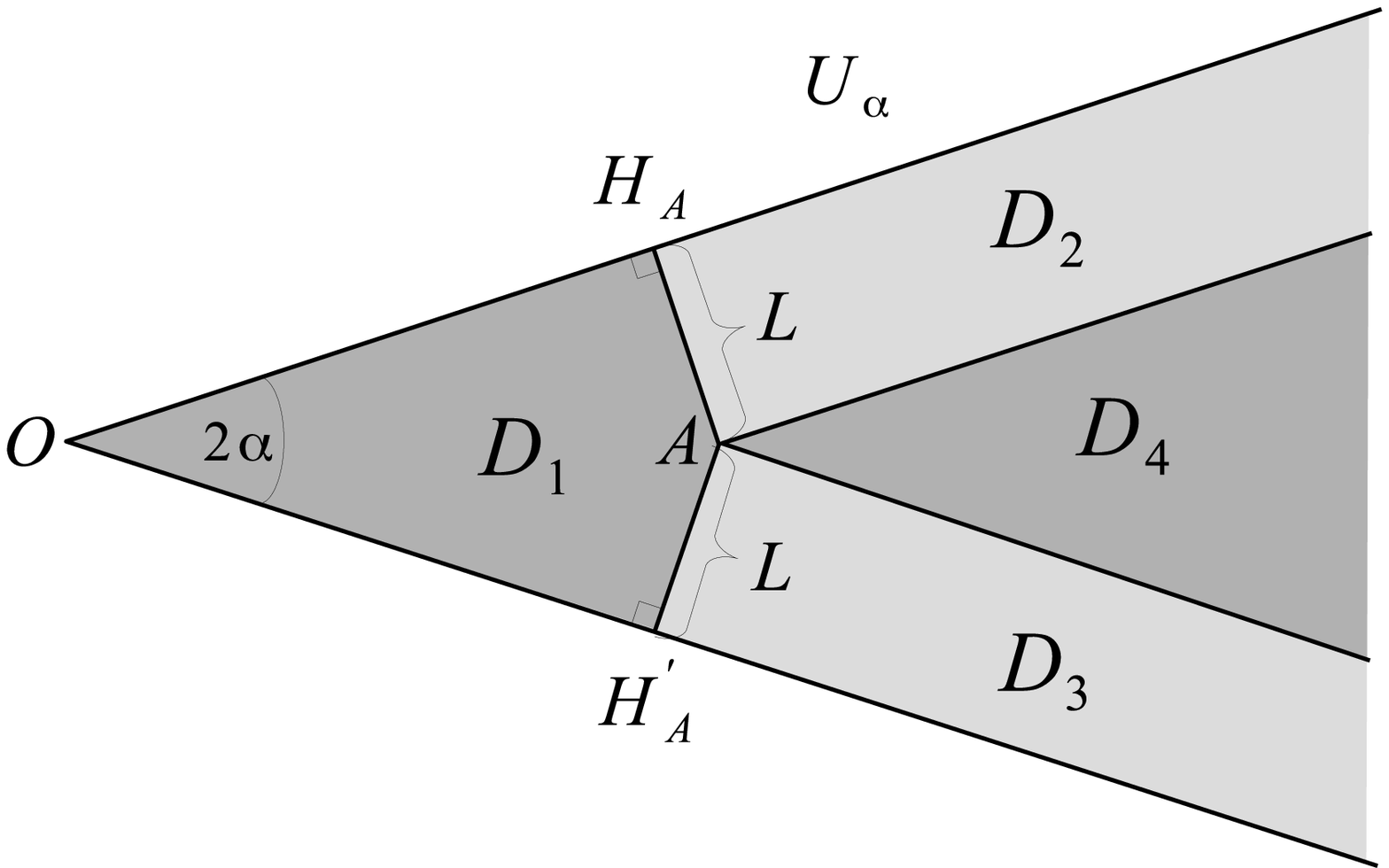}
	\caption{\label{dessinSpecess}Partition of $U_\alpha$.}
\end{figure}
Define for $j\in\{1,2,3,4\}$ the sesquilinear forms
\[
q_j(u,u)=\int_{D_j}\lvert \nabla u\rvert^2 dx -\gamma \int_{\partial U_\alpha\cap \partial D_j} \lvert u \rvert^2 ds,
\quad 
D(q_j)=H^1(D_j). 
\]
By the min-max principle, the inclusion $H^1(U_\alpha)\subset H^1(D_1 \mathop{\cup}D_2\mathop{\cup}D_3\mathop{\cup}D_4)$
implies the inequality
\begin{equation}
 \label{eqess}
\inf \Specess T^\gamma_\alpha \ge \inf \Specess Q,
\end{equation}
 where $Q$ is the self-adjoint operator acting in $L^2(U_\alpha)$
and associated to the sesquilinear form defined for $u\in H^1( D_1 \mathop{\cup}D_2\mathop{\cup}D_3\mathop{\cup}D_4)$ by
$q(u,u)=q_1(u,u)+q_2(u,u)+q_3(u,u)+q_4(u,u)$.
(Remark that we use the convention $\inf\varnothing=+\infty$.)
We have then $Q=Q_1\oplus Q_2 \oplus Q_3 \oplus Q_4$, where each $Q_j$
is the self-adjoint operator acting in $L^2(D_j)$ and
associated to $q_j$, and \eqref{eqess} implies
\begin{equation}
 \label{eqess2}
\inf \Specess T^\gamma_\alpha\ge \min_{j\in\{1,2,3,4\}} \inf \Specess Q_j.
\end{equation}
The operator $Q_1$ has a compact resolvent, then its essential spectrum is empty and $\inf\Specess Q_1=+\infty$, and the operator $Q_4$ is non-negative,
hence, $\inf \Specess Q_4\ge 0$.

The operator $Q_2$ et $Q_3$ are unitarily equivalent, so we have to study $Q_2$ only. By applying an anticlockwise rotation by angle $\theta=\frac{\pi}{2}-\alpha$
we see that $Q_2$ is unitarily equivalent to the self-adjoint operator $\tilde Q_2$
acting in $H^1\big((0,L)\times \mathbb R_+ \big)$ 
and associated with the sesquilinear form
\[
\tilde q_2(u,u)=\int_{0}^{+\infty} \int_0^L\lvert \nabla u \rvert ^2 dx
-\gamma \int_0^{+\infty} \big\lvert u (0,x_2)\big\rvert^2 dx_2,
\quad
u\in H^1\big((0,L)\times \mathbb R_+ \big).
\]
Clearly, $\tilde Q_2=T^\gamma_{RN}\otimes 1 + 1\otimes T_N$, where $T^\gamma_{RN}$ is the Robin-Neumann Laplacian acting in $L^2(0,L)$
on the domain
\[
D(T^\gamma_{RN})=\big\{u\in H^2(0,L):  u'(0)+\gamma u(0)=u'(L)=0\big\},
\]
and $T_N$ is the Neumann Laplacian in $L^2(\mathbb R_+)$. As $\Spec T_N=[0,+\infty)$, we have
 $\Specess Q_2=\Specess Q_2=\big[\Lambda_1(\Spec T^\gamma_{RN}),+\infty\big)$,
and an easy computation, see e.g. \cite[Lemme A.1]{ref2}, shows that
$\Lambda_1(T^\gamma_{RN})=-\gamma^2+o(1)$  as $L\to +\infty$.
Hence, for any $\epsilon>0$ there exists $L_\epsilon>0$ such as 
\[
\inf\Specess Q_2\geq -\gamma^2-\epsilon
\text{ for } L\geq L_\epsilon.
\]
As \eqref{eqess2} is valid for any $L>0$, we have
$\inf\Specess T^\gamma_\alpha\ge-\gamma^2-\varepsilon$ for any $\varepsilon>0$, which
implies \eqref{Specess0}.

\subsection{Description of the essential spectrum}
Due to the results of subsections \ref{b3} and \ref{b4}
we have $\Specess T^\gamma_\alpha\subset[-\gamma^2,+\infty)$
for any $\alpha$.
In order to conclude the proof of the point (a), it is sufficient to show the inclusion
\begin{equation}
\label{Specess02}
[-\gamma^2,+\infty)\subset \Spec T^\gamma_\alpha.
\end{equation}
It is more convenient to work with the rotated sector
\[
\tilde U_\alpha =\big\{x\in\mathbb R^2: 0<\arg(x_1+ix_2)<2\alpha \big\}
\]
and the associated Robin Laplacian $\tilde T^\gamma_\alpha$ in $L^2(\tilde U_\alpha)$
corresponding to the sesquilinear form
\[
\tilde t^\gamma_\alpha(u,u)=\int_{\tilde U_\alpha} |\nabla u|^2 dx -\gamma\int_{\partial \tilde U_\alpha} |u|^2ds,
\quad
u\in H^1(\tilde U_\alpha).
\]
Clearly, $\tilde T^\gamma_\alpha$ is unitarily equivalent to $T^\gamma_\alpha$.
The proof of \eqref{Specess02} consists in finding, for each $k\in\mathbb R$,
a family of functions $f_N\in D(\tilde T^\gamma_\alpha)$ with
\begin{equation}
\label{limiteTF}
\frac{\lVert \tilde T^\gamma_\alpha f_N-(k^2-\gamma^2)f_N\rVert^2_{L^2(\tilde U_\alpha)}}{\lVert f_N\rVert^2_{L^2(\tilde U_\alpha)}} \to 0 \text{ for } N\to +\infty,
\end{equation}
then the result follows by the spectral theorem for self-adjoint operators.
Let $\phi \in C^\infty(\mathbb R)$ satisfying $\phi(t)=0$ for $t\leq 0$ and $\phi(t)=1$ for $t \geq 1$.
For $N\in \mathbb N$, define a function $f_N:\tilde U_\alpha\to \mathbb C$ by
\[
f_N(x_1,x_2)=e^{ikx_1-\gamma x_2} \phi(2N-x_1)\phi(x_1-N) \phi(aN-x_2),
\]
where $a=\sin 2\alpha$ for $\alpha<\frac\pi 4$ and $a>0$ is arbitrary otherwise.
The function $f_N$ is smooth, compactly supported and satisfies the Robin boundary condition
$\partial f_N/\partial \nu=\gamma f_N$ at $\partial\tilde U_\alpha$, hence,
it belongs to $D(\tilde T^\gamma_\alpha)$, and $\tilde T^\gamma_\alpha f_N=-\Delta f_N$.
Easy estimates show that for large $N$ we have
\[
\big\|f_N\big\|^2_{L^2(\Tilde U_\alpha)}\ge c N \text{ with } c>0,
\quad
\big\| -\Delta f_N-(k^2-\gamma^2)f_N\big\|^2_{L^2(\tilde U_\alpha)}=O(1),
\]
which gives \eqref{limiteTF}.

Finally, the combination of \eqref{Specess02} with the result of subsection~\ref{b3}
gives the point (c) of the theorem.

\section{Study of the operator $H^{\infty}_a$}\label{appb}

We consider the operator acting on $L^2(0,+\infty)$ defined by:
\[
H_a=-\frac{d^2}{dr^2} -\frac{1}{4r^2}-\frac{1}{a r},
\quad
D(H_a)=C^\infty_c(\mathbb R_+).
\]
and the associated sesquilinear form $h_a$,
\[
h_a(u,u)=\int_{\mathbb R_+}
\Big(
\big| u'(r)\big|^2dr-\frac{\big|u(r)\big|^2}{4r^2}-\frac{\big|u(r)\big|^2}{a r}
\Big) dr,
\quad
u\in C^\infty_c(\mathbb R_+).
\]
Some parts of the analysis of the operator $T^\gamma_\alpha$
are based on the spectral properties of the Friedrichs extension
$H^\infty_a$ of $H_a$. Various parts of the description
and of the spectral analysis 
of $H^\infty_a$ are spread through the literature,
and in the present section we give a compact presentation
of the necessary results.

\subsection{The adjoint of $H_a$}
Let  $H_a^*$ be the adjoint of $H_a$. Recall the deficiency subspaces of $H_a$ are defined by
\[
\mathcal K_{\pm}=\Ker(H_a^*\mp i)=\Ran(H_a\pm i)^{\perp},
\]
and their dimensions $n_{\pm}=\dim(\mathcal K_\pm)$ are called the deficiency indices of $H_a$.
The following proposition gives us the existence of self-adjoint extensions of $H_a$.
\begin{proposition}
The operator $H_a^*$ is given by the same differential expression as $H_a$ and acts on the domain
\begin{equation}
         \label{dha}
D(H_a^*)=\big\{u\in L^2(\mathbb R_+):\, u\in H^2(\epsilon,+\infty) \text{ for any } \epsilon>0 \text{ and }
H_a^*u\in L^2(\mathbb R_+)\big\}.
\end{equation}
The deficiency indices of $H_a$ are equal to $1$, and $H_a$ admits self-adjoint extensions.
\end{proposition}

\begin{proof}
Denote by $\mathcal D$ the set on the right-hand side of \eqref{dha}. Let $v\in D(H^*)$, then for all $u\in D(H_a)$ one has
\[
\langle H_au,v\rangle_{L^2(\mathbb R_+)}=\int_{\mathbb R_+}\overline{\left(-u''(r)-\frac{1}{4r^2}u(r)-\frac{1}{ a r} u(r)\right)} v(r) dr=\langle u,H_a^*v\rangle_{L^2(\mathbb R_+)},
\]
and
\[
H_a^*v(r)=-v''(r)-\frac{1}{4r^2}v(r)-\frac{1}{a r} v(r) \quad \text{in}\quad \mathcal D'(\mathbb R_+).
\]
In particular, $D(H_a^*)\subset\{v\in L^2(\mathbb R_+), H_a^*v\in L^2(\mathbb R_+)\}$. Let $u\in L^2(\mathbb R_+)$ satisfy
$H_a^*u\in L^2(\mathbb R_+)$, then, for all $\epsilon >0$, we have $u/r^2 \in L^2(\epsilon,+\infty)$ and $u/r \in L^2(\epsilon,+\infty)$.
Due to
\[
-u''(r)=H_a^*u(r)+\frac{1}{4r^2}u(r)+\frac{1}{a}u(r) \in L^2(\epsilon,+\infty),
\]
we also have $u\in H^2(\epsilon,+\infty)$ for all $\epsilon>0$.
This shows the inclusion $D(H_a^*)\subset \mathcal D$.

Let us prove the reverse inclusion. Let $v\in \mathcal D$. After an integration by parts we have,
for all $\epsilon>0$ and for all $u\in D(H_a)$, 
\begin{gather*}
\langle H_a u,v\rangle_{L^2(\mathbb R_+)}=\int_0^{\epsilon}\overline{H_a u(r)}v(r) dr+\int_{\epsilon}^{+\infty}\overline{u(r)}\left(-v''(r)-\frac{1}{4r^2}v(r)-\frac{1}{a r}v(r)\right)dr \\
+\overline{u'(\epsilon)}v(\epsilon)-\overline{u(\epsilon)}v'(\epsilon).
\end{gather*}
Notice that the boundary terms are well defined as $u,v,u',v'$ are continuous.
Moreover,  as $u\in C^\infty_c(\mathbb R_+)$ we have
\[
\lim_{\epsilon \to 0}\Big(\overline{u'(\epsilon)}v(\epsilon)-\overline{u(\epsilon)}v'(\epsilon)\Big)=0.
\]
Hence, in the limit $\epsilon \to 0$ we obtain
\[
\langle H_a u,v\rangle_{L^2(\mathbb R_+)}=\int_{\mathbb R_+}\overline{u(r)}\left(-v''(r)-\frac{1}{4r^2}v(r)-\frac{1}{a r}v(r)\right)dr,
\]
for all $u\in D(H_a)$. Then $\mathcal D\subset D(H_a^*)$ and, finally, $D(H_a^*)=\mathcal D$.

Recall that a symmetric operator admits self-adjoint extensions if and only if his deficiency indices are equal.
In our case, $H^*_a$  commutes with the complex conjugation, which gives immediately 
$n_+=n_-$, as $\mathcal K_+$ and $\mathcal K_-$ are mutually complex conjugate.

It remains to determine the deficiency indices. The functions $u\in\Ker(H_a^*+i)$
are the solutions to the differential equation
\[
-u''-\frac{1}{4r^2}u-\frac{1}{a r} u +iu=0.
\]
Represent $u(r)=w\big(2e^{i\frac{\pi}{4}r}\big)$, then the function $w$ is a solution to the Whitakker's equation,
\[
w''(y)+\left(\frac{1}{4y^2} +\frac{1}{2ae^{i\frac{\pi}{4}} y } -\frac{1}{4} \right ) w(y)=0.
\]
whose linearly independent solutions are expressed in terms on the confluent hypergeometric functions,
and the space of $L^2$ solutions is one-dimensional, see \cite[Eq.~13.1.31]{abs}.
\end{proof}

\subsection{Self-adjoint extensions of $H_a$}
To describe the self-adjoint extensions of the operator $H_a$ we use the boundary triple approach,
see e.g. \cite{ref6,GG}. The integration by parts gives  
\begin{gather*}
\langle H_a^*\phi,\psi\rangle_{L^2(\mathbb R_+)}-\langle\phi,H_a^*\psi\rangle_{L^2(\mathbb R_+)}=\delta(\phi,\psi),
\quad \psi,\phi \in D(H_a^*),\\
\text{with }
\delta(\phi,\psi)=\lim_{r\to 0+}\left ( \overline{\phi'(r)}\psi(r)-\overline{\phi(r)}\psi'(r)\right).
\end{gather*}
The self-adjoint extensions of $H_a$ are restrictions of $H^*_a$ to maximal subspaces $D$ satisfying $D(H_a)\subset D \subset D(H_a^*)$
with the property that $\delta(\phi,\psi)=0$ for all $\phi,\psi \in D$. Notice that $\delta(\phi,\psi)$ is finite but $\lim_{x\to 0}\phi(x)$ or $\lim_{x\to 0}\phi'(x)$ could be infinite.

As shown in e.g. \cite[page 301]{ref7}, based on the asymptotic behavior
of the confluent hypergeometric functions, there exists linear maps
$a_1,a_2:D(H_a^*)\to\mathbb C$  such that
for any $\psi \in D(H_a^*)$ one has, as $x\to 0$,
\begin{align*}
\psi(r)&=a_1(\psi) \sqrt{r} +a_2(\psi)\sqrt{r} \ln r+O(r^{\frac{3}{2}}\ln r), \\
\psi'(r)&=\frac{a_1(r)}{2\sqrt{r}}+ \frac{a_2(\psi)}{\sqrt{r}}\left(\frac{\ln r}{2}+1 \right) +O(\sqrt{r} \ln r),
\end{align*}
or, equivalently,
\[
a_2(\psi)=\lim_{r\to 0^+} \frac{\psi(r)}{\sqrt{r}\ln r}, \quad
a_1(\psi)= \lim_{r\to 0^+} \frac{\psi(r)-a_2(\psi)\sqrt{r} \ln r }{\sqrt{r}},
\]
and a direct computation gives the equality
$\delta(\phi,\psi)= \overline{a_2(\phi)}a_1(\psi)-\overline{a_1(\phi)}a_2(\psi)$.
Furthermore, for any $(b_1,b_2)\in\mathbb C^2$ there is $\psi\in D(H_a^*)$
such that $a_1(\psi)=b_1$ and $a_2(\psi)=b_2$. In the language of \cite{ref6},
the triple $(\mathbb C,a_1,a_2)$ is a boundary triple
for $H_a$, and any self-adjoint extension
of $H_a$ is a restriction of $H_a^*$
to the functions $\psi$ satisfying the boundary
condition $a_1(\psi)\cos\vartheta=a_2(\psi)\sin\vartheta$,
where $\vartheta$ is a real-valued parameter.

By \cite[Theorem 3.1]{ref8}, 
the Friedrichs extension $H^\infty_a$
corresponds to the boundary condition $a_2(\psi)=0$, i.e. to $\vartheta=\frac\pi 2$.
The spectral properties for this case are completely analyzed in \cite[Subsection 8.3.3]{ref7}:
the essential spectrum is
$[0,+\infty)$, and the discrete spectrum consists of the simple eigenvalues
$\mathcal E_n(a)$ with the associated eigenfunctions $\psi_n$ given by
\[
\mathcal E_n(a)=-\frac{1}{a^2}\frac{1}{(2n-1)^2}, \quad
\psi_n(r)=\sqrt r \exp\Big(-\frac{r}{(2n-1)a}\Big)L_{n-1}\Big(\dfrac{2 r}{(2n-1)a}\Big),
\quad n\in\mathbb N,
\]
where $L_m$ are the Laguerre polynomials \cite[Section 22.1]{abs}.

\section{Weak quasimodes}\label{appd}

At reader's convenience we provide a complete proof
of the following simple assertion.

\begin{proposition}\label{propd1}
Let $T$ be a self-adjoint operator in a Hilbert space $\mathcal H$ such that
$T\ge a>0$, and  let $t$ be the associated sesquilinear form.
Assume that there exist a non-zero $u\in D(t)$ and numbers $\lambda>0$
and $\varepsilon\in(0,1)$ such that
\begin{equation}
  \label{eq-weak}
\big|\,
t(u,v)-\lambda \langle u, v\rangle 
\big|
\le \varepsilon \sqrt{t(u,u)}\sqrt{t(v,v)}
\text{ for all } v\in D(t),
\end{equation}
then $\dist\big(\lambda,\Spec T\big)\le \dfrac{\varepsilon}{1-\varepsilon}\,\lambda$.
\end{proposition}

\begin{proof}
Due to the spectral theorem for self-adjoint operators we have $D(t)=D(\sqrt T)$, and for $f,g\in D(t)$ there holds
$t(f,g)=\langle \sqrt T f,\sqrt T g\rangle$. Hence, we can rewrite the inequality \eqref{eq-weak} as
\[
\Big|
\langle \sqrt T u , \sqrt T v\rangle-\lambda \big\langle T^{-1} \sqrt Tu, \sqrt T v\big \rangle 
\Big|
\le \varepsilon\big\|\sqrt T u\big\|\cdot \big\|\sqrt T v\big\|,
\] 
and,
\begin{equation}
     \label{eq-weak2}
\bigg|
\Big\langle
(1-\lambda T^{-1}) \sqrt T u, \dfrac{\sqrt T v}{\|\sqrt T v\|} \Big\rangle
\bigg| \le \varepsilon \big\|\sqrt T u\big\|
\text{ for all } v\in D(t), \quad v\ne 0.
\end{equation}
As the vectors $\sqrt T v$ cover  the whole of $\mathcal  H$ as $v$ runs through $D(t)$,
taking the supremum over $v$ in \eqref{eq-weak2} gives
$\big\|(1-\lambda T^{-1}) \sqrt T u
\big\|
\le \varepsilon \big\|\sqrt T u\big\|$.
As $\sqrt T u\ne 0$, by the spectral theorem
for self-adjoint operators
we have $\dist\big(1,\Spec (\lambda T^{-1})\big)\le \varepsilon$, which means
that there exists $\mu\in \Spec T\subset[a,+\infty)$ such that $|1-\lambda \mu^{-1}|\le\varepsilon$.
Hence, $|\mu-\lambda|\le \varepsilon \mu$.  In particular,
$\mu-\lambda\le \varepsilon \mu$ and $\mu\le \lambda/(1-\varepsilon)$, which concludes the proof.
\end{proof}

\section{Applications to $\delta$-interactions on star graphs}\label{sec6}

We say that a subset $\Gamma\subset\RR^2$ is a \emph{star graph}
if it can be obtained as the union of finitely many rays starting at
the origin. Let $(r,\theta)$ be the polar coordinate systemcentered at the origin,
then $\Gamma$ can be identified with a family $(\theta_1,\dots,\theta_M)$ with
$0\le \theta_1<\dots<\theta_M<2\pi$ by $\Gamma:=\bigcup_{j=1}^M \big\{(r,\theta):\theta=\theta_j,\, r\ge 0\big\}$.

By the $\delta$-interaction of strength $\gamma>0$ one means
the Schr\"odinger operator formally written as $Q_{\Gamma,\gamma}=-\Delta-\gamma \delta_\Gamma$,
where $\delta_\Gamma$ is the Dirac $\delta$-distribution supported by $\Gamma$,
which is defined as the unique self-adjoint operator in $L^2(\RR^2)$ associated with
the closed lower semibounded sesqualinear form
\[
q_{\Gamma,\gamma}(u,u)=\int_{\RR^2} |\nabla u|^2dx - \gamma \int_\Gamma |u|^2 ds,
\quad u\in H^1(\RR^2),
\]
where $ds$ is the one-dimensional Hausdorff measure.
It seems that the operators of the above type were
first analyzed in~\cite{en1,en2}, and they are used to model to so-called
leaky quantum graphs, see e.g. the review \cite{exrev}.
The basic spectral properties of $Q_{\Gamma,\gamma}$ are well known:
the essential spectrum is equal to $[-\gamma^2/4,+\infty)$ and the discrete spectrum is non-empty
except for the degenerate cases $M=1$ and for $M=2$ with $\theta_2-\theta_1=\pi$, see \cite{ei,p17}.
On the other hand, the finiteness of the discrete spectrum was not discussed, except in some obvious
cases derived from the separation of variables. Using the above estimates obtained for Robin Laplacians
we are able to prove the following result for the star graphs:

\begin{theorem}
For any $M\ge 1$ and  $(\theta_1,\dots,\theta_M)$ with
$0\le \theta_1<\dots<\theta_M<2\pi$, the discrete spectrum of the associated operator $Q_{\Gamma,\gamma}$
is finite. Furthermore, if $M=2$ and $\pi/3\le (\theta_2-\theta_1)\mod 2\pi <\pi$, then the discrete spectrum consists
of a single eigenvalue.
\end{theorem}

\begin{proof}
Recall that by $N(A,\lambda)$ we denote the number of discrete eigenvalues, counting multiplicities, of a self-adjoint operator $A$
in $(-\infty,\lambda)$. Therefore, we need to show that $N(Q_{\Gamma,\gamma},-\gamma^2/4)<\infty$.

The proof is by comparison with Robin Laplacians in sectors. As noted above, the result trivially holds for $M=1$.
Assume that $M\ge 2$ and consider the $M$ infinite sectors $V_j$ given in the polar coordinates $(r,\theta)$
by
\[
V_j=\{ (r,\theta): \theta\in (\theta_j,\theta_{j+1}),\, r>0 \},
\quad
j=1,\dots,M,\quad \theta_{M+1}:=2\pi+\theta_1.
\]
Define sesquilinear forms $q_j$ in $L^2(V_j)$ by
\[
q_j(u_j,u_j):=\int_{V_j} |\nabla u_j|^2dx - \dfrac{\gamma}{2} \int_{\partial V_j} |u_j|^2ds,
\quad u_j\in H^1(V_j),
\]
and let $Q_j$ be the associated self-adjoint operators in $L^2(V_j)$, which are in fact the Robin Laplacians in $V_j$
with the Robin coefficient $\gamma/2$. The self-adjoint operator $Q$ in $\bigoplus_{j=1}^M L^2(V_j)\simeq L^2(\RR^2)$
associated with the sesquilinear form
\[
q\big((u_1,\dots u_M),(u_1,\dots u_M)\big)=\sum_{j=1}^M q_j(u_j,u_j), \quad u_j \in H^1(V_j)
\]
is then represented as $Q=Q_1\oplus\dots\oplus Q_M$. As $Q_j$ is unitary equivalent to $T^{\gamma/2}_{(\theta_{j+1}-\theta_j)/2}$,
its essential spectrum is $[-\gamma^2/4,+\infty)$, and $N(Q_j,-\gamma^2/4)<\infty$, as follows
from respectively Theorem~\ref{SPEC} and Theorem~\ref{thm31}.
Now remark that the restriction $u_j$ of any $u \in H^1(\RR^2)$ to any $V_j$
belongs to $D(q_j)$, and one has the equality
\[
q_{\Gamma,\gamma}(u,u)=q\big((u_1,\dots u_M),(u_1,\dots u_M)\big), \quad u\in D(q_{\Gamma,\gamma}).
\]
By the min-max principle it follows that
\begin{equation}
  \label{eq-n4}
N(Q_{\Gamma,\gamma}, -\gamma^2/4)\le \sum_{j=1}^M N(Q_j, -\gamma^2/4).
\end{equation}
As each summand on the right-hand side is finite, the first part of the assertion follows.

It remains to prove the second part of the assertion. Remark that, as mentioned above, we know already that
$N(Q_{\Gamma,\gamma},-\gamma^2/4)\ge 1$, so we only need to prove the reverse inequality.
By applying a suitable rotation we may assume that $\pi/3\le \theta_2-\theta_1 <\pi$,
then we have $N(Q_1,-\gamma^2/4)=1$ by Theorem~\ref{thmpi6} and $N(Q_2,-\gamma^2/4)=0$ by Theorem~\ref{SPEC}(c),
and the substitution into \eqref{eq-n4} gives the result.
\end{proof}

\end{document}